\documentclass[11pt, a4paper]{article}

\usepackage [english, russian] {babel}
\usepackage{amssymb}
\usepackage{amsfonts}
\usepackage{amsthm}
\usepackage{amsmath}
\usepackage{eucal}
\usepackage{upgreek}
\usepackage{amscd}
\usepackage[all]{xy}
 
\newcommand{\Aut}{\mbox{{\rm Aut}}}
 
\newtheorem{thm}{Theorem}[section]
\newtheorem{pro}[thm]{Proposition}
\newtheorem{lem}[thm]{Lemma}
\newtheorem{cor}[thm]{Corollary}

\newtheorem{ex}[thm]{Example}

\newtheorem{df}[thm]{Definition}
 
\newtheorem{rem}[thm]{Remark}

\newcommand{\st}{\mbox{{\rm St}}}

\newcommand{\ib}{\mbox{{\rm ib}}}

\newcommand{\St}{\mbox{{\rm St}}}
\newcommand{\aut}{\mbox{{\rm Aut}}}

\begin{document}
\selectlanguage{english}

\title{The Roelcke Precompactness and Compactifications of Transformations Groups of Discrete Spaces and Homogeneous Chains}

\author{B. V. Sorin}

\date{}

 \maketitle
 
 \begin{abstract}
 The Roelcke precompactness of transformation groups of discrete spaces and chains in the permutation topology and LOTS  in the topology of pointwise convergence is studied. For ultratransitive actions compactifications of transformation groups using the Ellis construction are built.
 \end{abstract}
 
\section{Introduction and preliminary remarks} 

Studying "big" topological groups, we face the problem of evaluating their  "smallness" in a sense \cite{Pestov}. In particular, one is to determine a uniformity on a group compatible with a group structure, which would be totally bounded. The Roelcke uniformity is the lower uniformity on the group (the greatest lower bound of the right and left uniformities) \cite{RD}. 

V. Uspenskij initiated the search of Roelcke-compactifications of topological transformation groups by considering homeomorphisms as points of hyperspaces (graphs in a squared space) \cite{8,U2001,U2008}. T. Tsankov \cite{Tsan} gave characterization of Roelcke precompact subgroups of permutation groups of countable discrete spaces $X$ in the permutation topology (pointwise convergence topology where $X$ is discrete) using group oligomorphism. 

An approach to the question when, under the action, the Roelcke precompactness of the acting group follows from total boundedness of the maximal equiuniformity on a phase space can be found in \cite[Proposition 9.17]{RD} and \cite[Proposition 6.4]{Megr}. The construction of an enveloping Ellis semigroup (for transformation groups on which the topology of pointwise convergence is an admissible group topology) is used in  \cite{Sorin}.  This construction allows connecting the maximal equiuniformity on a phase space with the Roelcke uniformity via an intermediate equiuniformity determined by a system of small subgroups introduced in \cite{Kozlov}. The latter uniformity coincides with the Roelcke uniformity in the case of the permutation topology on a group.

The work studies the Roelcke precompactness of transformation groups of discrete spaces and chains in the permutation topology and LOTS  in the topology of pointwise convergence. For ultratransitive actions compactifications of transformation groups using the Ellis construction are built.

The main techniques used to obtain results of the work are presented in \S2. 
R. Ellis proposed a construction of building a compactification of a transformation group of a compactum, which is {\it an enveloping Ellis semigroup}~{\rm\cite{Ellis}} (a semigroup with multiplication continuous on the right and continuous on the left for continuous mappings). In {\rm\cite{Sorin}} the Ellis construction was used for building compactifications of transformation groups of compacta, for which the topology of pointwise convergence was an admissible group topology, as well as for finding a sufficient condition of their Roelcke precompactness. In \S2  the Ellis construction extends to the homeomorphism group $G$ of a non-compact space $X$. The condition of the compactness of the $G$-space $X$ is replaced by the condition of $X$ is $G$--Tykhonoff, i.e. the presence of an equiuniformity on $X$.

A sufficient condition, under which a uniformity on a group generated by its embedding in the product of uniform spaces (the usage of the Ellis construction) and the uniformity built on a family of small subgroups (point stabilizers) \cite[\S4]{Kozlov} coincide, is found in Theorem \ref{Roelcke precomp2-1}. The fact that the latter uniformity is totally bounded ensures the Roelcke precompactness of a group \cite[Corollary 4.5]{Kozlov}. The combination of these two facts allows obtaining a sufficient condition of the Roelcke precompactness of groups in Corollary \ref{Roelcke precomp2-2}. 

The topology of pointwise convergence is the smallest admissible group topology. This fact allows studying the question of the Roelcke precompactness of a transformation group in the permutation topology (Corollary \ref{top_precomp}). 

The Roelcke precompactness of subgroups of permutations groups of discrete spaces is studied in \S3. The equality of Roelcke uniformity and uniformity built on the family of point stabilizers is established (Proposition \ref{coin}).

Theorems \ref{Roelcke precomp3-1} and \ref{Roelcke precomp3-2} present sufficient conditions and a criterion for the Roelcke precompactness of subgroups of permutations groups of discrete spaces in the permutation topology using maximal equiuniformities on phase spaces and their connection with oligomorphism of actions. 

Theorem \ref{Roelcke precomp3-3}  provides a complete solution to the question of the Roelcke precompactness of automorphism groups of simple chains. It is restated as follows: 

\noindent {\bf Corollary \ref{Roelcke precomp3-4}}. Let $X$ be a simple chain.

(1) $X$ is rigid $\Longleftrightarrow$ the group $\aut(X)$ is not Roelcke precompact in any admissible group topology for its action on the corresponding $X$ homogeneous GO-spaces. 
 
(2) $X$ is ultrahomogeneous $\Longleftrightarrow$ the group $(\aut(X), \tau_{\partial})$ is Roelcke precompact $\Longleftrightarrow$ the group $(\aut(X), \tau_p)$ is Roelcke precompact.

(3) The group $(\aut(X), \tau_{\partial})$ is Roelcke precompact iff the group $(\aut(X), \tau_p)$ is Roelcke precompact.  

The Roelcke precompactness of ultratransitive subgroups of an automorphism group of ultrahomogeneous (and cyclic) chains in the topology of pointwise convergence (the permutation topology in terms of this paper) is proved in \cite[Proposition 6.6]{Megr}. Examples \ref{Roelcke precomp3-5} show the possibilities of using results from \S3.

A sufficient condition of equality between the Roelcke uniformity and the uniformity obtained using the Ellis construction (Theorem \ref{Roelcke precomp4-1}) is given in \S4. This result made it possible to show (Corollary \ref{Roelcke precomp4-2}) that Roelcke compactification of ultrahomogeneous transformation group in permutation topology is an enveloping Ellis semigroup.  Corollary \ref{Roelcke precomp4-3} describes the (Roelcke) compactifications of automorphism groups of ultrahomogeneous chains which are enveloping Ellis semigroups in the case of continuous chains. It uses compactifications of their phase spaces outlined in Lemma \ref{comp}. The construction of compactifications of chains from \cite{Fed} and the construction of the transition from a GO-space to a linearly ordered space from \cite{Miwa} are used. The proof of point (2) of Corollary \ref{Roelcke precomp4-3} is an extension of Theorem 3 \cite{Sorin} from a compact case to a case of ultrahomogeneous chains.

See \cite[\S4]{8} for more information about algebraic structures on Roelcke compactifications. 

The results of studying the Roelcke precompactness of automorphism groups of chains that are not simple are presented in \S5. Theorem \ref{LexO_Aut} describes the structure of automorphism groups of chains that are not simple. They are semidirect (topological) products. Theorem 8 from \cite{Sorin} is used. 

In Theorems \ref{Roelcke precomp5-2} and \ref{Roelcke precomp5-3} characterizations of the Roelcke precompactness of automorphism groups of chains that are not simple are given using their structure. 

From the constructed inverse spectrum in the proof of Theorem \ref{Roelcke precomp5-3} it is clear that the topology of pointwise convergence on the automorphism group of a LOTS corresponding to the chain that is not simple and does not contain simple proper regular intervals is "approximated" by permutation topologies on the automorphism groups of its quotient spaces (by the equivalence relation that is determined by a regular interval). 

Corollaries \ref{coin2} and \ref{Roelcke precomp3-4} solve the problem of the equivalence of the condition of the Roelcke precompactness of groups 
$(\aut (X), \tau_p)$ and $(\aut(X), \tau_{\partial})$ in the case of homogeneous chains that are either simple or have a simple proper regular interval. 

In preliminary remarks, the consideration of homogeneous GO-spaces provides motivation for the study of the permutation topology. Maximal equivariant compactifications of ultrahomogeneous LOTS and discrete chains are also constructed. 

The author is grateful to Prof. K. L. Kozlov, Prof. V. G. Pestov and Prof. M. G. Megrelishvili for providing useful information related to the subject of the study. 

Non-empty sets, (topological) Hausdorff spaces are considered. The uniformities on the space are compatible with its topology. Terminology and notations from \cite{Engelking} and \cite{RD} are used. $\mathbb Q$ are rational, $\mathbb P$ are irrational, and $\mathbb R$ are real numbers. For a family $\Omega$ of subsets $X$ and $Y\subset X$, $\Omega\wedge Y=\{O\cap Y\ |\ O\in\Omega\}$. $N_G(e)$ is a family of open neighbourhoods of the unit of the topological group $G$.

\subsection{Basic properties of Roelcke precompactness}
All the necessary information about the Roelcke uniformity $L\wedge R$ on a group $G$ can be found in \cite{RD}. A group  $G$ is Roelcke precompact if the Roelcke uniformity is totally bounded. Let us recall the main properties of Roelcke precompactness of topological groups.

{\bf Facts 1.} (1) A dense subgroup  $H$ of a group  $G$ is Roelcke precompact iff the group $G$ is Roelcke precompact {\rm\cite[Proposition 3.24]{RD}} and {\rm\cite[Proposition 2.2]{Tsan}}. Moreover, Roelcke compactifications of $H$ and $G$ are isomorphic \cite[Corollary 8.3.11]{Engelking}.

(2) An open subgroup of a Roelcke precompact group is Roelcke precompact {\rm\cite[Proposition 3.24]{RD}}.

(3) A continuous homomorphic image of a Roelcke precompact group is a Roelcke precompact group {\rm\cite[Proposition 2.2]{Tsan}}.

(4) If the normal subgroup $H$ and the factor group $G/H$ of the group $G$ are Roelcke precompact, then the group $G$ is Roelcke precompact {\rm\cite[Proposition 2.2]{Tsan}}.

(5) The inverse limit of inverse spectrum of Roelcke precompact groups and homomorphisms is Roelcke precompact {\rm\cite[Proposition 2.2]{Tsan}}.

(6) A product of topological groups is Roelcke precompact iff the factors are Roelcke precompact groups {\rm\cite[Proposition 3.35]{RD}}.

If two group topologies $\sigma\leq\tau$ are given on the group $G$, then from setting the bases of the right, left, two-sided, Roelcke uniformities it follows that $R_{\sigma}\subset R_{\tau}$, $L_{\sigma}\subset  L_{\tau}$,  $(L\vee R)_{\sigma}\subset  (L\vee R)_{\tau}$,  $(L\wedge R)_{\sigma}\subset  (L\wedge R)_{\tau}$ for uniformities on a group in the corresponding topologies. 
 
We consider infinite cardinals $\kappa$. The smallest cardinal $\kappa$ such that the uniformity $\mathcal U$ has a basis of coverings of cardinality no more than $\kappa$ is called the narrowness index $\ib(\mathcal U)$ of the uniformity $\mathcal U$. The concept was introduced by I. Guran \cite{Gur}, and it is called index of boundedness in \cite[Ch. 1, \S1]{Borub}. The uniformity $\mathcal U$ is totally bounded if there exists a basis of finite coverings.
 
\begin{lem} \label{top_uniform}
If two group topologies $\sigma\leq\tau$ are given on the group $G$, then  

{\rm(1)} $\ib (R_{\sigma})\leq\ib (R_{\tau}), \ib (L_{\sigma})\leq\ib (L_{\tau}), \ib (L\vee R)_{\sigma}\leq\ib (L\vee R)_{\tau}, \ib (L\wedge R)_{\sigma}\leq\ib (L\wedge R)_{\tau}.$
 
{\rm (2)} From the totally boundedness of uniformity $R_{\tau}$ {\rm(}respectively $L_{\tau}$,  $(L\vee R)_{\tau}$,  $(L\wedge R)_{\tau}${\rm)} it follows that the uniformity $R_{\sigma}$ {\rm(}respectively  $L_{\sigma}$,  $(L\vee R)_{\sigma}$,  $(L\wedge R)_{\sigma}${\rm)} is totally bounded.~\hfill $\square$
\end{lem}

\subsection{Topologization of a transformation group}
The action of the group $G$ on the set $X$ is called effective if the kernel of the action $\{g\in G\ |\ g(x)=x, \forall x\in X\}$ is the unit of $G$. If $G$ effectively acts on $X$, then $G\subset {\rm S}(X)$, where ${\rm S}(X)$ is the permutation group of $X$. 

The subgroup ${\rm St}_{x_1,\ldots, x_n}=\{g\in G\ |\ g(x_i)=x_i, i=1,\ldots, n\}$ is a stabilizer of the group $G$.

If $X$ is a topological space, ${\rm Hom}(X)$ is the group of its homeomorphisms, then the group $G$ effectively acts on the space $X$ if $G\subset {\rm Hom}(X)$. If $X$ is a discrete space, then ${\rm Hom}(X)={\rm S}(X)$.

A topology in which a group is a topological group and its action $G\curvearrowright X$ is continuous is called an admissible group topology \cite{Arens} on the group $G$ effectively acting on the topological space $X$. 

For an effective action of the group $G$ on the discrete space $X$, the permutation topology $\tau_{\partial}$, the subbase of the neighbourhoods of whose unit is formed by open-closed subgroups, i.e. stabilizers of points, is the smallest admissible group topology. The group $G$ in the permutation topology is non-Archimedean (a unit neighbourhood base is formed by open-closed subgroups). 

If topology of pointwise convergence $\tau_p$ (the subbase is formed by sets of the form $[x,O]=\{f\in G\ |\ f(x)\in O\}$, $O$ is open in $X$) is an admissible group topology on the group $G$ of homeomorphisms of the space $X$ acting effectively, then it is the smallest admissible group topology \cite[Lemma 3.1]{Kozlov}, $\tau_{\partial}\geq\tau_p$, $\tau_{\partial}$ is an admissible group topology and  $\tau_{\partial}=\tau_p$, if $X$ is a discrete space.

From Lemma \ref{top_uniform} we have  
 
\begin{cor}\label{top_precomp} 
If $(G, \tau_{\partial})$ is Roelcke precompact, then  $(G, \tau_{p})$ is Roelcke precompact.
 
If $(G, \tau_p)$ is not Roelcke precompact, then the group $G$ in any admissible group topology is also not Roelcke precompact.\hfill $\square$
\end{cor}

\begin{rem}
Generally speaking, the Roelcke precompactness of $(G, \tau_{p})$ does not imply the Roelcke precompactness of $(G, \tau_{\partial})$.
\end{rem}

Indeed, the action of the (infinite) topological group $G$ on itself by multiplication on the left is uniformly equicontinuous with respect to the left uniformity $L$ and the topology of pointwise convergence $\tau_{p}$ is an admissible group topology on $G$ coinciding with the topology of $G$ (see, for example, \cite[Example 3.6]{Kozlov}). If $G$ is compact, then $(G, \tau_{p})$ is Roelcke precompact. However, in the permutation topology, $G$ is discrete (and infinite). Hence, it is not Roelcke precompact.

\subsection{Ultratransitive action} 
\begin{df} The action of the group $G$ on the set $X$ is strongly $n$-transitive, $n\geq 1$, if for any families of distinct $n$ points $x_1, \dots, x_n$ and $y_1, \dots, y_n$ there exists $g\in G$ such that $g(x_k)=y_k$, $k=1,\ldots, n$. 

The action $G\curvearrowright X$, which is strongly $n$-transitive for all $n\in\mathbb N$, is called ultratransitive.
\end{df}

{\bf Facts 2.} (1) A group $G$ which acts ultratransitively on $X$ is a dense subgroup of $({\rm S}(X), \tau_{\partial})$. 

Indeed, let the set $O$ be open in $({\rm S}(X), \tau_{\partial})$ and $g\in O$. Then there are $x_1,\ldots, x_n\in X$ such that $g{\rm St}_{x_1, \dots, x_n}\subset O$ and $g{\rm St}_{x_1, \dots, x_n}=\{h\in G\ |\ h(x_i)=g(x_i), i=1,\ldots, n\}$ an open neighbourhood of $g$. Since $G$ acts ultratransitively on $X$
there is $h\in G$ such that $h(x_i)=g(x_i)$, $i=1,\ldots, n$. Evidently, $h\in O$.\hfill $\Box$

(2) The subgroup ${\rm S}_{<\omega}(X)$ of the group ${\rm S}(X)$ whose elements has finite supports acts ultratransitively on $X$. Hence, ${\rm S}_{<\omega} (X)$ is a dense subgroup of $({\rm S}(X), \tau_{\partial})$. 

(3) The Roelcke precompactness of the group $({\rm S}(X), \tau_{\partial})$ is proved in \cite{Gau}  (see also \cite[Example 9.14]{RD}), 
the Roelcke precompactness of the group $({\rm S}_{<\omega}(X), \tau_{\partial})$ is proved in \cite{Ban}. 

From Fact 1 (1) it follows that any group $G$ which acts ultratransitively on $X$ is Roelcke precompact. Moreover, their Roelcke compactifications are isomorphic to the Roelcke compactification of  $({\rm S}(X), \tau_{\partial})$.

(4) The space $X$ is ultrahomogeneous if the action of its homeomorphism group ${\rm Hom}(X)$ is ultratransitive. Ultrahomogeneous spaces are locally compact metrizable CDH spaces whose complement to any finite subset is connected. These include the spheres $S^{n-1}$ in Euclidean spaces $\mathbb R^n$, $n\geq 3$, the Hilbert cube $Q$ (see, for example, \cite{ArM}). 

The group of homeomorphisms of an ultrahomogeneous space with the permutation topology is Roelcke precompact, although the permutation topology is not necessarily admissible. The homeomorphism groups of the spheres $S^{n}$, $n\geq 2$, and the Hilbert cube $Q$ with the compact open topology (the smallest admissible group topology) are not Roelcke precompact \cite{Ros}. Hence, by Lemma \ref{top_uniform}, and with any admissible group topology, these groups are not Roelcke precompact.

\subsection{Homogeneous chains} 
If $X$ and $Y$ are chains, then their product $X\times Y$ with lexicographic order is denoted by $X\otimes_{\ell} Y$; their concatenation is denoted by $X\lozenge Y$ (on the disjoint union of $X$ and $Y$, the linear order is as follows: $x<y$ if $x\in X$, $y\in Y$, linear order restricted on $X$ and $Y$ coincide with linear orders on $X$ and $Y$, respectively).

The subset $Y$ of the chain $X$ is called an interval if for any $x\leq y\in Y$ and $x\leq z\leq y\Longrightarrow z\in Y$. The intervals are: 

\noindent half-intervals $[a, b)=\{x\in X\ |\ a\leq x<b\}$, $(a, b]=\{x\in X\ |\ a< x\leq b\}$; 

\noindent open intervals $(a, b)=\{x\in X\ |\ a< x< b\}$, $(a, \to)=\{x\in X\ |\ a< x\}$, $(\gets, b)=\{x\in X\ |\ x< b\}$, the set $X$; 

\noindent segments $[a, b]=\{x\in X\ |\ a\leq x\leq b\}$ (in particular, points of $X$). 

\begin{df} Let $X$ be a chain, $\aut(X)$ be a group of order-preserving bijections {\rm(}automorphisms{\rm)} of $X$.
 
$X$ is called a homogeneous chain if the action $\aut(X)\curvearrowright X$ is transitive {\rm(}i.e., for any $x,y\in X$ there exists $f\in\aut(X)$ such that $f(x)=y${\rm)}.
\end{df}

{\bf Facts 3.} (1) A homogeneous chain is either single-point or infinite.
 
(2) A homogeneous chain $X$ is either discrete (for any $ x\in X$ there exists $x<x^+$ and $(x, x^+)=\emptyset$ \cite{KM}), or dense ($X$ is dense if for any $x<y\in X$ there exists $z\in (x, y)$ \cite{KM}).

\begin{df} The interval $J$ of a homogeneous chain $X$ is called regular {\rm\cite[Definition 5]{Ohkuma2}} if 
$$\forall x, y\in J, \forall g\in\aut(X) ((gx\in J)\Longrightarrow (gy\in J)).$$
 
A homogeneous chain $X$ is called simple {\rm\cite[Definition 6]{Ohkuma2}} if $X$ has no proper regular intervals {\rm(}being non-empty intervals, non-proper are either single-point or all $X${\rm)}. The group $\aut(X)$ in this case is called o-primitive {\rm\cite{Glass}}.
 
A chain $X$ is called 2-homogeneous if for any pairs of points  $x<y$ and $x'<y'$ there exists $g\in\aut(X)$ such that $g(x)=x'$, $g(y)=y'$. The group $\aut(X)$ in this case is called o-2-transitive {\rm\cite{Glass}}. 
 
A homogeneous chain $X$ is called rigid {\rm\cite{GHHS}} if for any $x,y\in X$ there exists an only $g\in\aut(X)$ such that $g(x)=y$. The group $\aut(X)$ in this case is called regular or uniquely transitive {\rm\cite{Glass}}.
\end{df}

{\bf Facts 4.} (1) A 2-homogeneous chain is dense. 

(2) A 2-homogeneous chain is ultrahomogeneous {\rm\cite[Lemma 1.10.1]{Glass} (see also \cite{Ovch})}, i.e. for any families of different $n$ points $x_1<\ldots< x_n$ and   $y_1<\ldots< y_n$ there exists $g\in\aut(X)$ such that $g(x_k)=y_k$, $k=1,\ldots, n$, $n\in\mathbb N$. 

(3) A proper regular interval $J$ of a homogeneous chain $X$ is a homogeneous interval, a chain and $X=Y\otimes_{\ell} J$, where $Y$ is a homogeneous chain \cite[Theorem 7]{Ohkuma2}.

In \cite[point  2, point 3.5]{Ohkuma2}, \cite{Glass} and \cite{Hol} it is established that the following theorem holds: 
 
\begin{thm}\label{Classific_gr} For a homogeneous chain $X$, the following conditions are equivalent: 

{\rm(1)} the set $X$ is simple;

{\rm(2)}  the set $X$ 

\qquad{\rm(i)} 2-homogeneous {\rm (}ultrahomogeneous{\rm )}, or

\qquad{\rm(ii)} is rigid and is a subgroup of $\mathbb R$ {\rm(}isomorphic to $\aut(X)${\rm)};

{\rm(3)} the group $\aut(X)$ is o-primitive;

{\rm(4)} the group $\aut(X)$

\qquad{\rm(i)} is o-2 transitive, or

\qquad{\rm(ii)} is uniquely transitive and is a subgroup of the abelian group $\mathbb R$ {\rm(}their full description is given in {\rm\cite{GHHS})}. \hfill $\square$
\end{thm} 

\subsection{Topologization of the automorphism group of a homogeneous chain}
A {\it generalized ordered space}, or GO-space is a chain with a topology stronger than the linear order topology, whose  base is formed by the intervals \cite{Lut}.
  
\begin{lem} In the following topologies, a homogeneous infinite chain $X$ is a GO-space and each element $\aut(X)$ is a homeomorphism: 

{\rm(1)} topology of linear order $\tau$ {\rm(}the topology base is intervals $(x, y)$, $x<y\in X$, in this case $X$ is a linearly ordered space {\rm (LOTS))};

{\rm(2)} "arrow" topology $\tau_{\to}$ or $\tau_{\gets}$ {\rm(}topology base is half-intervals $[x, y)$ {\rm(}right arrow{\rm)} or half-intervals $(x, y]$, {\rm(}left arrow{\rm)},   $x<y\in X${\rm)};

{\rm(3)} discrete topology $\tau_d$.

$$\tau\leq\left\{\begin{array}{c}\tau_{\to} \\ \tau_{\gets}\end{array}\right\}\leq\tau_d.$$
$\tau=\left\{\begin{array}{c}\tau_{\to} \\ \tau_{\gets}\end{array}\right\}=\tau_d\Longleftrightarrow$ $X$ is a discrete homogeneous chain.  
\end{lem}
 
\begin{proof} Let $\sigma$ be an arbitrary topology on $X$ in which $X$ is a GO-space. If $X$ is discrete, then  $\tau=\left\{\begin{array}{c}\tau_{\to} \\ \tau_{\gets}\end{array}\right\}=\tau_d$, since for any $x\in X$ there are $x^-<x< x^+$ and the intervals $(x^-, x)$, $(x, x^+)$ are empty sets.
 
If $X$ is not discrete, then let $x$ be an arbitrary point of $X$. Intervals containing $x$ can be either open intervals $(a, b)$, $a<x<b$, or half-intervals $[x, b)$, $x<b$,  $(a, x]$, $a<x$. The cases of segments $[a, b]$, $[x, b]$, $[a, x]$ are reduced to the previous cases, since the topology on a GO-space is stronger than the topology of linear order. 
 
If only open intervals are considered as open neighbourhoods of the point $x$, then they form the basis of a linear order topology at the point $x$, and, due to the homogeneity of $X$, a linear order topology will be set on $X$.  
 
If at least one half-interval $[x, b)$ (respectively $(a, x]$) is added to the open intervals as open neighbourhood of the point $x$, then they form the topology base at the point $x$ of the form $\{[x, b)\ |\ b>x\}$ (respectively   $\{(a, x]\ |\ a<x\}$), and, due to the homogeneity of $X$, the "arrow" topology $\tau_{\to}$ (respectively $\tau_{\gets}$) will be set on $X$.
 
If at least one half-interval $[x, b)$ and at least one half-interval $(a, x]$ are added to the open intervals as open neighbourhoods of the point $x$, then they form the base of a discrete topology at the point $x$ and, due to the homogeneity of $X$, a discrete topology will be set on $X$. 
 
The relations $\tau\leq\left\{\begin{array}{c}\tau_{\to} \\ \tau_{\gets}\end{array}\right\}\leq\tau_d$ obviously hold.
 
The implication $\tau=\tau_d\Longrightarrow$ $X$ is a discrete homogeneous chain follows from the existence of $x^-<x$ and $x<x^+$ such that $(x^-, x)$, $(x, x^+)$ are empty sets. The inverse implication is established at the beginning of the proof.
 
Since in an order-preserving bijection, the images of intervals and half-intervals are intervals and half-intervals, respectively, then the elements of $\aut(X)$ are homeomorphisms. 
\end{proof}
 
\begin{cor} 
{\rm (1)} On a discrete homogeneous chain, the topology in which $X$ is a homogeneous GO-space is unique. $X$ is discrete {\rm LOTS}.

{\rm (2)} $\tau_{\to}<\tau_d \Longleftrightarrow \tau_{\gets}<\tau_d.$

{\rm (3)} $\left\{\begin{array}{c}\tau_{\to} \\ \tau_{\gets}\end{array}\right\}=\tau_d \Longleftrightarrow \tau_d=\tau.$ 

{\rm (4)} If there is an order-reversing bijection on $X$, then $(X, \tau_{\to})$ and $(X, \tau_{\gets})$ are homeomorphic.\hfill $\square$
\end{cor}

\begin{pro}\label{perm_least_adm_top}  Let $X$ be a homogeneous chain. 

{\rm (1)}
On the group $\aut(X)$, the topology of pointwise convergence $\tau_{p}$ is the smallest admissible group topology for the action $\aut(X)\curvearrowright (X, \tau)$. The topology $\tau_{\partial}$ is an admissible group topology and $\tau_{\partial}\geq \tau_p$. If $X$ is discrete, then $\tau_{\partial}=\tau_p$.

{\rm (2)} On the group $\aut(X)$, the permutation topology $\tau_{\partial}$ is the smallest admissible group topology for actions $\aut(X)\curvearrowright (X, \tau_{\to})$, $\aut(X)\curvearrowright (X, \tau_{\gets})$, $\aut(X)\curvearrowright (X, \tau_{d})$.
\end{pro}

\begin{proof}  $(\aut(X), \tau_{\partial})$ is a topological group \cite{RD}. 

The statement (1) of the proposition is proved in \cite{Ovch} and \cite{Sorin}. 

(2) It is easy to verify that the permutation topology $\tau_{\partial}$ is an admissible group topology for the actions $\aut(X)\curvearrowright (X, \tau_{\to})$, $\aut(X)\curvearrowright (X, \tau_{\gets})$, $\aut(X)\curvearrowright (X, \tau_{d})$. In the latter case, it is obviously the smallest. 

Let $\sigma$ be an admissible group topology, for example, on the group $\aut(X, \tau_{\to})$. For any point $x$ and its neighbourhood $[x, y)$, $x<y$, there is a neighbourhood $O$ of the unit of the group, and a neighbourhood of the point $x$ of the form $[x, x')$, $x<x'$, such that $O[x, x')\subset [x, y)$. Then for any homeomorphism $g$ from the neighbourhood $O\cap O^{-1}$ of the unit of the group we have $g(x)\in [x, y)$ and $g^{-1}(x)\in [x, y)$. If, for example, $g(x)=z>x$, then $g^{-1}(x)<g^{-1}(z)=x$. Hence, $g(x)=x$ for any $g\in O\cap O^{-1}$ and $O\cap O^{-1}\subset\st_x$. Consequently, $\sigma\geq\tau_{\partial}$ and $\tau_{\partial}$ is the smallest admissible group topology on $\aut(X)$ for the action $\aut(X)\curvearrowright (X, \tau_{\to})$.
\end{proof}

\begin{rem} $(\aut(X), \tau_{\partial})$ is a subgroup $({\rm S}(X), \tau_{\partial})$.
\end{rem}

\subsection{$G$-spaces}
All the necessary information can be found in \cite{Megr0}.

Under the continuous action $G\curvearrowright X$ of the topological group $G$ on the space $X$, the triple $(G,X,\curvearrowright)$ is called a $G$-space (abbreviated $X$ is a $G$-space).

The uniformity $\mathcal U$ on $X$ is called equiuniformity if the action of $G\curvearrowright X$ is saturated (i.e. any homeomorphism from $G$ is uniformly continuous) and is bounded (i.e. for any $u\in\mathcal U$ there are $O\in N_G(e)$ and $v\in\mathcal U$ such that the covering $\{OV\ |\ V\in v\}$ refines $u$). In this case, $X$ a $G$-Tikhonoff space, the completion of $X$ with respect to a totally bounded equiuniformity, is $G$-compactification or equivariant compactification of $X$, and there exists $\beta_G X$ --- the maximal $G$-compactification of $X$ which corresponds to the maximal totally bounded  equiuniformity. 

\subsection{Maximal $G$-compactifications of spaces of ultrahomogeneous chains}\label{M-G-comp}
Let $X$ be an ultrahomogeneous chain, $\aut(X)$ its automorphism group. 

The base of the maximal equiuniformity $\mathcal U_X^{p}$ for the action $(\aut(X), \tau_{p})\curvearrowright (X, \tau)$  is formed by finite coverings 
$$\{Ox\ |\ x\in X\}=(\gets, y_2)\cup (y_1, y_2)\cup (y_1, y_4)\cup\ldots\cup(y_{2n-3}, y_{2n})\cup (y_{2n-1}, y_{2n})\cup (y_{2n-1}, \to),$$
for  $y_1<x_1<y_2<\ldots<y_{2n-1}<x_n<y_{2n}$, $n\in\mathbb N$, $O=[x_1, (y_1, y_2)]\cap\ldots\cap [x_n, (y_{2n-1}, y_{2n})]$
(the corresponding diagonal entourage ${\rm U}_{O}$), 
$\beta_{p}X$ is the maximal $G$-compactification of $X$ (completion of $X$  with respect to uniformity $\mathcal U_X^{p}$ \cite[Theorem 3]{ChK2}).

According to the description in \cite{Fed}, the smallest linearly ordered compactification $m (X, \tau)$ \cite{Kauf}  
is generated by replacing each gap in $X$ by a point with a natural continuation of the order. At the same time, $m (X, \tau)$ is the only linearly ordered compactification to which the action of $\aut(X)$ is continuously extended (discontinuity if the gap is replaced by two points); the action $(\aut(X), \tau_{p})\curvearrowright m (X, \tau)$ is continuous, $m (X, \tau)$ is connected. 

The base of the maximal equiuniformity $\mathcal U_X^{\partial}$  for the action $(\aut(X), \tau_{\partial})\curvearrowright (X, \tau_d)$  is formed by the finite disjoint coverings 
$$\{{\rm St}_{x_1, \dots, x_n}x\ |\ x\in X\}=(\gets, x_1)\cup\{x_1\}\cup (x_1, x_2)\cup\{x_2\}\cup\ldots\cup(x_{n-1}, x_n)\cup\{x_n\}\cup (x_n, \to),$$ $x_1<\ldots<x_n\in X,$
(the corresponding entourage of the diagonal ${\rm U}_{x_1, \dots, x_n}$), 
$\beta_{\partial}X$ is the maximal $G$-compactification of $X$ (completion of $X$  with respect to the uniformity  $\mathcal U_X^{\partial}$ \cite[Theorem 3]{ChK2}).

A discrete space $(X, \tau_d)$ is a GO-space. There is the smallest LOTS 
$$(X, \tau_d)\otimes_{\ell}\{-1, 0, 1\},$$ 
in which $X$ is a dense subspace \cite{Miwa}, and $(X,\tau_d)\otimes_{\ell}\{-1, 0, 1\}$ is naturally embedded in any other linearly ordered extension of $X$, in which $X$ is dense. In this case, the action $(\aut(X), \tau_{p})\curvearrowright \big((X, \tau_d)\otimes_{\ell}\{-1, 0, 1\}\big)$ is continuous. The smallest linearly ordered compactification of $(X, \tau_d)$ is $m\big((X, \tau_d)\otimes_{\ell}\{-1, 0, 1\}\big)$, which is zero-dimensional.

\begin{rem} A chain without proper gaps is called continuous~{\rm\cite[Ch.\ 6, \S\ 2]{KM}}. There are  unique linearly ordered compactification of $(X, \tau_d)$ or $(X, \tau)$, where $X$ is a continuous chain. There are obtained by addition of two points $\sup$ and $\inf$ of $(X, \tau_d)\otimes_{\ell}\{-1, 0, 1\}$ or $X$ respectively.
\end{rem}

\begin{lem}\label{comp}

{\rm (1)} $\beta_{p}X=m (X, \tau)$,

{\rm (2)} $\beta_{\partial}X=m\big((X, \tau_d)\otimes_{\ell}\{-1, 0, 1\}\big):=m (X, \tau_d)$.
\end{lem}

\begin{proof} (1) In any open covering of a connected linearly ordered compactum $m(X,\tau)$, one can refine a finite covering $\Omega$ from the intervals $I_1=(\gets, b_1),\ldots, I_k=(a_k, \to)$, such that $a_2<b_1< a_3<b_2<\ldots<a_k<b_{k-1}\in X$ (it is enough to choose a minimal system from the intervals in any covering in the sense that any of its subsystems is not a covering). Then in $\Omega\wedge X$, the trace of $\Omega$ on $X$,  the covering 
$$(\gets, b_1)\cup (a_2, b_1)\cup (a_2, b_2)\cup\ldots\cup(a_{k-1}, b_{k-1})\cup (a_k, b_{k-1})\cup (a_k, \to),$$
which belongs to $\mathcal U_X^{p}$ is refined, if an arbitrary point is selected in the intervals $(a_j, b_{j-1})$, $j=2,\ldots, k$. Since any covering of $\mathcal U_X^{p}$ continues till the covering of $m(X,\tau)$, then $\beta_{p}X=m (X, \tau)$.

(2) As in the case of (1), in any open covering of a zero-dimensional compactum $m\big((X,\tau_d)\otimes_{\ell}\{-1, 0, 1\}\big)$ it is possible to refine a finite covering of open-closed intervals. By introducing the order on the intervals (at their left ends), we obtain a disjoint covering from the open-closed intervals 
$I_1=(\gets, b_1),\ldots, I_k=(b_k, \to)$, subtracting sequentially from the $j$ interval the union of the previous ones. Let us correct the latter covering by removing endpoints of the form $(\{x\}, 0)$ from the intervals (if any), and adding them as single-point open-closed intervals to the corrected interval system to obtain an open covering $\Omega$ of the space $m\big((X, \tau_d)\otimes_{\ell}\{-1, 0, 1\}\big)$. Then $\Omega\wedge X$, its trace on $X$, obviously coincides with the covering of $\mathcal U_X^{\partial}$. Since any covering of $\mathcal U_X^{\partial}$ continues to cover of  $m\big((X, \tau_d)\otimes_{\ell}\{-1, 0, 1\}\big)$, then  $\beta_{\partial}X=m (X, \tau_d)$.
\end{proof}

\begin{rem} It is possible to prove {\rm Lemma \ref{comp}}, using the uniqueness of the linearly ordered $G$-compactification by showing that the proximity corresponding to the maximal equiuniformities on $ (X, \tau)$ and $\big((X, \tau_d)\otimes_{\ell}\{-1, 0, 1\}, \tau\big)$ are ordered proximities, see, for example, {\rm \cite{MegrS}}.
\end{rem}
 
\section{The Ellis construction and equiuniformities\\ on a transformation group}\label{Ellis_constr}
We assume that the topology of pointwise convergence is an admissible group topology on the group $G$ of homeomorphisms of the space $X$ acting effectively. 
 
The homeomorphisms of the group $G$ as continuous maps can be identified with the points of the space $X^X$ (of degree $X$). Namely, the injective mapping $\imath: G\to X^X$, $\imath(g)=(g(x))_{x\in X}$, is defined. The topology on $X^X$ is initial with respect to projections on factors and its restriction on $\imath(G)$ is the topology of pointwise convergence on $G$.
 
On $X^X$, as on the product of copies of $X$ with the action $\alpha: G\times X\to X$ of the group $G$, the continuous action of the group $G$ is defined: $\alpha_{\Delta}: G\times X^X\to X^X, \alpha_{\Delta} (g, (t_x)_{x\in X})=(gt_x)_{x\in X}$. 
In this case, $\imath$ is an equivariant embedding of $G$ into $X^X$ (the group $G$ is considered a $G$-space with its action onto itself by multiplication on the left).
 
Suppose that $\mathcal U_X$ is the equiuniformity on the space $X$, $\mathcal U_{\Pi}$ is the equiuniformity on $X^X$ initial with respect to projections on factors, $\mathcal U$ is the restriction of the equiuniformity $\mathcal U_{\Pi}$ on $G=\imath(G)$. Closure of $\imath(G)$ in completion of $X^X$ with respect to $\mathcal U_{\Pi}$ is completion of $\imath(G)$ with respect to the uniformity $\mathcal U$. The base of the equiuniformity $\mathcal U$ is formed by coverings
$$\{U_{x_1, \dots, x_n; g; {\rm U}}=\{h\in G\ |\ (g(x_k), h(x_k))\in {\rm U}, k=1, \dots, n\}: g\in G\},$$
where $x_1,\dots, x_n\in X$, $n\in\mathbb N$, ${\rm U}$ is the entourage of the equiuniformity $\mathcal U_X$. 
 
The family $\mathcal K$ of stabilizers  
$${\rm St}_{x_1, \dots, x_n}=\bigcap\{{\rm St}_{x_k}\ |\ k=1, \dots, n\}$$
of points $x_1,\dots, x_n\in X$, $n\in\mathbb N$, is a directed family  $\mathcal K$ of small subgroups ($H'\leq H\Longleftrightarrow H'\subset H$) of the group $G$, which defines the equiuniformity $R_{\mathcal K}$ on $G$, whose base are coverings $\{OgH: g\in G\}$, $O\in N_G(e)$, $H\in \mathcal K$; $L\wedge R\subset R_{\mathcal K}\subset R$ \cite[\S4]{Kozlov}. An explicit type of coverings from the (possible) base is 
$$\{O_{x_1, \dots, x_n; {\rm U}}g{\rm St}_{x_1, \dots, x_n}\ |\ g\in G\},$$
where $x_1,\dots, x_n\in X$, $n\in \mathbb N$, $O_{x_1, \dots, x_n; {\rm U}}=\{h\in G\ |\ (x_k, h(x_k))\in {\rm U}, (x_k, h^{-1}(x_k))\in {\rm U}, k=1, \dots, n\}$,  ${\rm U}$ is the entourage of the uniformity $\mathcal U_X$. Note that $O^{-1}_{x_1, \dots, x_n; {\rm U}}=O_{x_1, \dots, x_n; {\rm U}}$.
 
If the uniformity $R_{\mathcal K}$ is totally bounded, then the group $G$ is Roelcke precompact \cite[Corollary 4.5]{Kozlov}. The following theorem generalizes Theorem 1 from \cite{Sorin}.
 
\begin{thm}\label{Roelcke precomp2-1}
 
{\rm (1)} $\mathcal U\subset R_{\mathcal K}$.
 
{\rm (2)} Let for any points $x_1,\dots, x_n\in X$, $n\in\mathbb N$, and for any entourage ${\rm U}\in \mathcal U_X$ there is an entourage ${\rm V}={\rm V}(x_1,\dots, x_n; {\rm U})\in \mathcal U_X$ such that the condition is met:

\noindent if for $g, h\in G$ $(g(x_k), h(x_k))\in {\rm V}$ holds, $k=1, \dots, n$, then there exists $g'\in g{\rm St}_{x_1, \dots, x_n}$ such that $h\in O_{x_1, \dots, x_n; {\rm U}}g'$.
 
\noindent Then $\mathcal U=R_{\mathcal K}$.
\end{thm}
 
\begin{proof}
(1) Any finite subproduct $X^n=X^{\{x_1,\dots, x_n\}}$ of the product $X^X$ is a $G$-space, the projection $\pi_n: X^X\to X^n$ is an equivariant mapping. The equiuniformity on $X^n$ is the product of equiuniformities $\mathcal U_X$. Therefore, in any uniform covering $u$ of any subproduct $X^n$ it is possible to refine a covering of the form $\{Ox\ |\ x\in X^n\}$ for some $O\in N_G(e)$. Then in the covering $\pi^{-1}_n u\wedge G=\{\pi^{-1}_nU\cap \ |\ U\in u\}$ the covering $\{\pi_n^{-1}(Ox)=O\pi_n^{-1}x\ |\ x\in X^n\}\wedge G$ is refined. 
 
For $x=(g(x_1),\dots, g(x_n))\in X^n$ $O\pi_n^{-1}x=Og{\rm St}_{x_1, \dots, x_n}$, therefore $\{O\pi_n^{-1}x\ |\ x\in X^n\}\wedge G=\{ Og{\rm St}_{x_1, \dots, x_n}\ |\ g\in G\}\in R_{\mathcal K}$. It remains to note that every covering from $\mathcal U$ has the form $\pi^{-1}_n u\wedge G$ by virtue of the definition of $\mathcal U=\mathcal U_{\Pi}|_{\imath(G)}$ and the initiality of $\mathcal U_{\Pi}$ with respect to projections $\pi_n$.
 
(2) To prove the equality $\mathcal U=R_{\mathcal K}$, by virtue of (1), it is sufficient to show that in the covering $$\{O_{x_1, \dots, x_n; {\rm U}}g{\rm St}_{x_1, \dots, x_n}\ |\ g\in G\},$$
where $x_1,\dots, x_n\in X$, $n\in \mathbb N$, $O_{x_1, \dots, x_n; {\rm U}}=\{h\in G\ |\ (x_k, h(x_k))\in {\rm U}, (x_k, h^{-1}(x_k))\in {\rm U}, k=1, \dots, n\}$,  ${\rm U}\in\mathcal U_X$,  it is possible to refine a covering 
$$\{U_{x_1, \dots, x_n; g; {\rm V}}=\{h\in G\ |\ (g(x_k), h(x_k))\in {\rm V}, k=1, \dots, n\}\ |\ g\in G\},$$
for some ${\rm V}\in\mathcal U_X$. 
 
By the condition for $x_1,\dots, x_n\in X$ and ${\rm U}\in\mathcal U_X$ there exists ${\rm V}\in\mathcal U_X$ such that if for $g, h\in G$ $(g(x_k), h(x_k))\in {\rm V}$, $k=1, \dots, n$, then there exists $g'\in g{\rm St}_{x_1, \dots, x_n}$ such that $h\in O_{x_1, \dots, x_n; {\rm U}}g'$. Then for any $h\in U_{x_1, \dots, x_n; g; {\rm V}}$ and any fixed $g\in G$ 
$$h\in O_{x_1, \dots, x_n; {\rm U}} g{\rm St}_{x_1, \dots, x_n}$$
and, therefore, $U_{x_1, \dots, x_n; g; {\rm V}}\subset   O_{x_1, \dots, x_n; {\rm U}}g{\rm St}_{x_1, \dots, x_n}$, i.e. the covering  
$\{U_{x_1, \dots, x_n; g; {\rm V}}\ |\ g\in G\}$ is refined in the covering of $\{O_{x_1, \dots, x_n; {\rm U}}g{\rm St}_{x_1, \dots, x_n}\ |\ g\in G\}$.
\end{proof}
 
\begin{cor} \label{Roelcke precomp2-2}
If $\mathcal U=R_{\mathcal K}$ and $\mathcal U_X$ is a totally bounded equiuniformity on $X$, then the group $G$ is Roelcke precompact.
\end{cor}
 
\begin{proof}
According to \cite[Theorem 4.2.]{Kozlov}, $L\wedge R\subset R_{\mathcal K}$. Therefore, from the total boundedness of $\mathcal U_X$ (and thus $\mathcal U$) and the equality $\mathcal U=R_{\mathcal K}$, it follows that the uniformity $L\wedge R$ is totally bounded. 
\end{proof}
 
\begin{rem}
The above construction of building extensions of transformation groups is used in {\rm\cite{Sorin}} for the group of homeomorphisms of the compactum $K$ in the topology of pointwise convergence. In this case, the only uniformity on $K$ is the equiuniformity. The resulting compactification is an enveloping Ellis semigroup {\rm\cite{Ellis}}. 
\end{rem}
 
\section{The Roelcke precompactness of subgroups $({\rm S}(X), \tau_{\partial})$} \label{R-prec_perm-top}
Let $G$ be a subgroup of the permutation group $({\rm S}(X), \tau_{\partial})$ of the discrete infinite space $X$. The base of the neighbourhoods of the unit of $G$ is formed by open-closed subgroups 
$${\rm St}_{x_1, \dots, x_n}, x_1,\dots, x_n\in X$$
and $G$ is non-Archimedean.
 
The base of the maximal equiuniformity $\mathcal U_X$ is formed by disjoint coverings of sets
$$\{{\rm St}_{x_1, \dots, x_n}x\ |\ x\in X\}, x_1,\dots, x_n\in X$$
(for the corresponding diagonal entourage ${\rm U}_{x_1, \dots, x_n}$) see, for example, \cite{ChK}, 
$\mathcal U$ is the equiuniformity on $G$, constructed in \S\ref{Ellis_constr}. 
 
The base of the Roelcke uniformity $L\wedge R$ on $G$ is formed by coverings
$$\{{\rm St}_{x_1, \dots, x_n}g{\rm St}_{x_1, \dots, x_n}\ |\ g\in G\}, x_1,\dots, x_n\in X,$$
which form the uniformity base of $R_{\mathcal K}$, constructed from a family of small subgroups $$\mathcal K=\{{\rm St}_{x_1, \dots, x_n}\ |\  x_1,\dots, x_n\in X\}.$$ Thus we have

\begin{pro}\label{coin}
$L\wedge R=R_{\mathcal K}$.\hfill$\square$
\end{pro}

\begin{df} The action of the group $G$ on the discrete space $X$ is oligomorphic if the correctly defined action $G\curvearrowright X^n$, $g(x_1, \ldots, x_n)=(gx_1, \ldots, gx_n)$, has a finite number of orbits, $n\in\mathbb N$.
\end{df}

\begin{thm}\label{Roelcke precomp3-1}
For the action $G\curvearrowright X$ on the discrete space $X$, the following conditions are equivalent:

{\rm (1)} the maximal equiuniformity $\mathcal U_X$ on $X$ is totally bounded;

{\rm (2)} the action of the group $G$ is oligomorphic;

{\rm (3)} the maximal equiuniformity $\mathcal U_{X^n}$ on $X^n$ is totally bounded for the action $G\curvearrowright X^n$, $n\in\mathbb N$.

If one of the equivalent conditions {\rm(1)--(3)} is met, then the group $G$ is Roelcke precompact.
\end{thm}

\begin{proof}  (1) $\Longrightarrow$ (2). We use the proof by induction.  The maximal equiuniformity $\mathcal U_X$ is totally bounded by the condition. So the action $G\curvearrowright X$ has a finite number of orbits.

Let the action $G\curvearrowright X^n$, $n\in\mathbb N$, has a finite number of orbits: $Y_1, \ldots, Y_k$, $y_1\in Y_1, \ldots, y_k\in Y_k$. Due to the total boundedness of $\mathcal U_X$ for any $j=1,\ldots, k$ the action ${\rm St}_{y_j}$ on $X$ has a finite number of orbits: $Z_{j1}, \ldots, Z_{jm}$, $z_1\in Z_{j1}, \ldots, z_m\in Z_{jm}$. We show that the orbits of the action $G\curvearrowright X^{n+1}$ are the sets $Y_j\times Z_{ji}$, $j=1,\ldots, k$, $i=1,\ldots, m$.

For the point $(y', x')\in X^{n}\times X$ let $y'\in Y_j$. There exists $g\in G$ such that $g(y_j)=y'$ (under the action $G\curvearrowright X^n$). Let $g^{-1}(y', x')=(y_j, x)$ (under the action $G\curvearrowright X^{n+1}$). There exist $z_i\in X$ and $h\in {\rm St}_{y_j}$ such that $h(z_i)=x$. Then $gh (y_j, z_i)=(y', x')$.

(2) $\Longrightarrow$ (1). Let the action of the group $G$ be oligomorphic, and the point $x=(x_1,\ldots, x_n)\in X^n$. If the orbits of action $G\curvearrowright X^{n+1}$ are the sets $Y_1, \ldots, Y_k$, then the orbits of action of the group  ${\rm St}_{x}$ on $X=\{x\}\times X$ are the sets $\{x\}\times X\cap Y_j$, $j=1,\ldots, k$. Thus, the uniformity $\mathcal U_X$ is totally bounded. 

(2) $\Longrightarrow$ (3). Since the unit neighbourhood base of the group $G$ is formed by subgroups of point stabilizers, and point stabilizers for actions 
$G\curvearrowright X^n$, $n\in\mathbb N$, allow natural identification (${\rm St}_{x_1,\ldots, x_m}$, $x_1,\ldots, x_m\in X^n$, coincides with the stabilizer of the coordinates of the points $x_1,\ldots, x_m$ under the action of $G\curvearrowright X$), then, first, the permutation topologies defined by the actions on $G$ coincide and, second, the maximal equiuniformity under the action of the subgroup ${\rm St}_{y_1,\ldots, y_m}$ on $X$ is totally bounded. From the equivalence of conditions (1) and (2), the action of the subgroup ${\rm St}_{x_1,\ldots, x_m}$ is oligomorphic and its action on $X^n$ has a finite number of orbits. Thus, the maximal equiuniformity $\mathcal U_{X^n}$ on $X^n$ under the action of $G\curvearrowright X^n$ is totally bounded.

The implication of (3) $\Longrightarrow$(1) is obvious.

To prove the last statement of the theorem, it is sufficient to show that any covering of $G$ of the form 
$$\{{\rm St}_{x_1, \dots, x_n}g{\rm St}_{x_1, \dots, x_n}\ |\ g\in G\}, x_1,\dots, x_n\in X$$
has a finite subcovering. 

Denote $x=(x_1,\ldots, x_n)\in X^n$, ${\rm St}_{x}={\rm St}_{x_1, \dots, x_n}$. It follows from condition (3) that for the subset 
$\{x\}\times Gx$ of the fiber $\{x\}\times X^n$ of the product $X^n\times X^n$ there exist $g_1,\ldots, g_m\in G$ such that
$$\{x\}\times Gx\subset\{x\}\times\big(\bigcup\{{\rm St}_{x}g_ix\ |\ i=1,\ldots, m\}\big).$$ Then for any $h\in G$ we have
$$(x, hx)=g(x, g_ix)=(gx, gg_ix)$$
for some $g\in {\rm St}_{x}$, and $i\in\{1, \ldots, m\}$. Hence, $g_i^{-1}g^{-1}h\in {\rm St}_{x}$, and $h\in {\rm St}_{x}g_i{\rm St}_{x}$.
\end{proof}

\begin{thm}\label{Roelcke precomp3-2} For the action $G\curvearrowright X$ on the discrete space $X$, the following conditions are equivalent:

{\rm (1)} the group $G$ is Roelcke precompact;

{\rm (2)} the maximal equiuniformity $\mathcal U_Y$ on $Y$ is totally bounded for the action $G\curvearrowright Y$ on any invariant subset $Y\subset X$, having a finite number of orbits;

{\rm (3)}  the restriction of action of the group $G$ on any invariant subset $Y\subset X$ having a finite number of orbits is oligomorphic;

{\rm (4)} the maximal equiuniformity $\mathcal U_{Y^n}$ on $Y^n$ is totally bounded for the action $G\curvearrowright Y^n$, $n\in\mathbb N$, where  $Y\subset X$ is an invariant subset having a finite number orbits.
\end{thm}

\begin{proof} (1) $\Longrightarrow$ (2). Take one point from each orbit of the action $G\curvearrowright Y$: $y_1,\ldots, y_k$. For an arbitrary neighbourhood $O\in N_G(e)$ there is a neighbourhood of the form $V={\rm St}_{x_1,\ldots, x_m, y_1,\ldots, y_k}\subset O$. Due to the Roelcke precompactness of the group $G$, there exists a finite set $g_1,\ldots, g_n\in G$ such that $\bigcup\{Vg_jV\ |\ j=1,\ldots, n\}=G$. Then for any $i=1,\ldots, k$ 
$$\bigcup\{Vg_jVy_i\ |\ j=1,\ldots, n\}=Gy_i~ \mbox{and}~ \bigcup\{Og_jy_i\ |\ j=1,\ldots, n, i=1,\ldots, k\}=Y.$$
That is, any covering $\{Oy\ |\ y\in Y\}$ has a finite subcovering and the uniformity $\mathcal U_Y$ is totally bounded. 

The equivalence of conditions (2), (3) and (4) is proved in Theorem \ref{Roelcke precomp3-1}.

(3) $\Longrightarrow$ (1). Let $Y=\bigcup\{Gx_j\ |\ j=1,\ldots, k\}\subset X$ be an invariant subset having a finite number of orbits for the action  $G\curvearrowright Y$.

The kernel $N$ of the action $G\curvearrowright Y$ is a closed normal subgroup  of $G$. The effective action of the factor group $G/N\curvearrowright Y$ is correctly defined. If we consider the group $H_Y=G/N$ in the permutation topology, then the action $H_Y\curvearrowright Y$ and the natural homomorphism $\varphi_Y:G\to H_Y$ are continuous, the maximal equiuniformities $\mathcal U|_{X}$ and $\mathcal U_Y$, generated by the action of $G$ and $H_Y$ on $Y$, respectively, coincide. The topological group $H_Y$ is Roelcke precompact by Theorem \ref{Roelcke precomp3-1}.

The family of invariant subsets of $X$ having a finite number of orbits for the action $G\curvearrowright Y$ forms an inclusion-directed set. If $Y'\subset Y$, then the homomorphism $\varphi_{YY'}:H_Y\to H_{Y'}$ (factorization by the action kernel) is defined, for which $\varphi_{YY'}\circ\varphi_Y=\varphi_{Y'}$. Thus, the inverse spectrum $\{H_Y, \varphi_{YY'}, Y\}$ from Roelcke precompact groups and homomorphisms is determined. Its inverse limit is a Roelcke precompact group (Fact 1(5)). 

Since the action of $G$ is effective, the family of surjective homomorphisms $\varphi_Y:G\to H_Y$ is a separating (points and closed sets) family of maps  (${\rm St}_{x_1, \dots, x_n}$ under the action $G\curvearrowright X$ contains the kernel of the action $G$ on $Y=\bigcup\{Gx_i\ |\ i=1,\ldots, n\}$ and the prototype of ${\rm St}_{x_1, \dots, x_n}$ under the action $H_Y\curvearrowright Y$, $x_1,\dots, x_n\in Y$). Thus, $G$ is a dense subgroup of the inverse limit  $\{H_Y, \varphi_{YY'}, Y\}$, and is Roelcke precompact (Fact 1 (1)). 
\end{proof}

\begin{thm} \label{Roelcke precomp3-3} 
Let $X$ be a simple chain. 

{\rm(1)} If $X$ is rigid, then the group $(\aut(X), \tau_p)$   {\rm(}and hence $(\aut(X), \tau_{\partial})${\rm)} is not Roelcke precompact. 

{\rm(2)} If $X$ is ultrahomogeneous, then the group $(\aut(X), \tau_{\partial})${\rm(}and hence $(\aut(X), \tau_p)${\rm)}  is Roelcke precompact.
\end{thm} 
 
\begin{proof} (1) By Theorem \ref{Classific_gr} $X$ is an (unbounded) subgroup of the abelian group $\mathbb R$. If $X$ is discrete, then $X$ is isomorphic to $\mathbb Z$, the group $\aut(X)$ in the topology $\tau_{\partial}=\tau_p$ on $\aut(X)$ is isomorphic to $\mathbb Z$ and is not Roelcke precompact. 
 
If $X$ is dense, then $X$ is a dense unbounded subset of $\mathbb R$, the group $\aut(X)$ in the topology $\tau_p$ is a dense unbounded subgroup of the abelian topological group $\mathbb R$ on which all group uniformities coincide and are not totally bounded.  Hence $(\aut(X), \tau_p)$ (and by Corollary \ref{top_precomp}  $(\aut(X), \tau_{\partial})$) is not Roelcke precompact. 
 
(2)  The base of the maximal equiuniformity $\mathcal U_{\partial}$ on the discrete space $(X,\tau_d)$ under the action $(\aut(X), \tau_{\partial})\curvearrowright (X, \tau_d)$ is formed by finite coverings
$$(\gets, x_1)\cup\{x_1\}\cup(x_1, x_2)\cup\{x_2\}\cup\dots\cup(x_{n-1}, x_n)\cup\{x_n\}\cup(x_n, \to), x_1<\ldots <x_n, n\in\mathbb N,$$
and the uniformity $\mathcal U_{\partial}$ is totally bounded. By  Theorem \ref{Roelcke precomp3-1}, the group $(\aut(X), \tau_{\partial})$ is Roelcke precompact.

By Corollary \ref{top_precomp} the group $(\aut(X), \tau_p)$ is Roelcke precompact. 
\end{proof}

From Proposition \ref{perm_least_adm_top} we have

\begin{cor}\label{Roelcke precomp3-4} Let $X$ be a simple chain.

{\rm (1)}  $X$ is rigid $\Longleftrightarrow$ the group $\aut(X)$ is not Roelcke precompact in any admissible group topology for its action on the corresponding $X$ homogeneous {\rm GO}-spaces. 

{\rm (2)}   $X$ is ultrahomogeneous $\Longleftrightarrow$ the group $(\aut(X), \tau_{\partial})$ is Roelcke precompact $\Longleftrightarrow$ the group $(\aut(X), \tau_p)$  is Roelcke precompact.

{\rm (3)}  The group $(\aut(X), \tau_{\partial})$ is Roelcke precompact iff the group $(\aut(X), \tau_p)$  is Roelcke precompact. \hfill $\square$
\end{cor} 
 
\begin{ex}\label{Roelcke precomp3-5}{\rm (1) $\mathbb Z$ is a rigid set.  The group $\Aut(\mathbb Z)$ is isomorphic to $\mathbb Z$, is $o$-primitive, uniquely transitive, and by Corollary \ref{Roelcke precomp3-4} is not Roelcke precompact in a discrete topology equal to $\tau_p=\tau_{\partial}$.

(2) $\mathbb Q, \mathbb P, \mathbb R$, $(0,1)$, $\mathbb Z\otimes_{\ell}\mathbb Q$ (isomorphic to $\mathbb Q$) and $\mathbb Z\otimes_{\ell}\mathbb P$ (isomorphic to $\mathbb P$) are ultrahomogeneous sets. Their automorphism groups by Corollary \ref{Roelcke precomp3-4} are Roelcke precompact in the topologies $\tau_{\partial}$ and $\tau_p$.

(3)  $\mathcal L=[0,\omega_1)\otimes_{\ell} [0,1)$ is a long ray, $\mathcal L_{-}$ is a long ray $\mathcal L$ with reverse linear ordering.

It is easy to check that $L=\mathcal L\setminus \{(0,0)\}\subset\mathcal L$, $L_{-}=\mathcal L_{-}\setminus \{(0,0)\}\subset\mathcal L_{-}$ and $\tilde L=L_{-}\lozenge\{0\}\lozenge L$ are ultrahomogeneous sets and, by Corollary \ref{Roelcke precomp3-4}, groups of their automorphisms $\Aut(\star)$ are Roelcke precompact in the topologies $\tau_{\partial}$ and $\tau_p$ (see. \cite[§3, p. 3.3]{Sorin} for the topology of pointwise convergence).

(4) The Sorgenfrey line $\mathbb S$ is an ultrahomogeneous chain that is a GO-space. 

By Proposition {\rm\ref{perm_least_adm_top}}, the permutation topology $\tau_{\partial}$ is the smallest admissible group topology on the group $\Aut(\mathbb S)$, and the group $(\Aut(\mathbb S), \tau_{\partial})$ is Roelcke precompact by Corollary {\rm\ref{Roelcke precomp3-4}}.

(5) For the group $\aut({\rm \bf D})$ of LOTS "two arrows" ${\rm \bf D}=\{(0, 1)\}\lozenge\big( (0, 1)\otimes_{\ell} (0, 1)\big)\lozenge \{(1, 0)\}$, the smallest admissible group topology is the topology of pointwise convergence $\tau_p$. $\tau_p=\tau_{\partial}$, since $$\st_{(x, i)}=[(x, 1), [(x, 1), \to)]\bigcap [(x, 0), (\gets, (x, 0)]], x\ne 0, 1, i=0,1.$$
It is easy to check that the maximal equiuniformity under the action $(\aut({\rm \bf D}), \tau_{\partial})\curvearrowright ({\rm \bf D}, \tau_d)$ is totally bounded. Hence by Theorem \ref{Roelcke precomp3-1} the group  $(\aut({\rm \bf D}), \tau_{\partial})$ is Roelcke precompact.

A different approach is possible. LOTS $D=(0,1)\otimes_{\ell}\{-1,1\}$ is a subspace of ${\rm \bf D}$ and the automorphism groups $\Aut (D)$ and $\Aut ({\rm \bf D})$ in permutation topologies are topologically isomorphic. It can be shown that the group $(\Aut(D), \tau_{\partial})$ ($\tau_{\partial}$ is the smallest admissible group topology) is topologically isomorphic to the group $(\Aut((0, 1), \tau_d), \tau_{\partial})$. Hence, the group $(\aut({\rm \bf D}), \tau_{\partial})$ is Roelcke precompact (point (2)).

(6) The group $\aut(\mathbb M)$ of the GO-space the "Michael  line" $\mathbb M$ is naturally identified with the automorphism groups $\aut(\mathbb P)$ and $\aut(\mathbb Q)$ of invariant subsets: irrational numbers $\mathbb P$ in discrete topology and rational numbers $\mathbb Q$ in linear order topology, respectively. 

The smallest admissible group topology on $\Aut(\mathbb P)$ is the permutation topology $\tau_{\partial M}$, the smallest admissible group topology on $\Aut(\mathbb Q)$ is the topology of pointwise convergence $\tau_{p M}$. It is easy to check that $\tau_{\partial M}\geq \tau_{p M}$. Therefore, the smallest admissible group topology on $\Aut(\mathbb M)$ will be $\tau_{\partial M}$, in which $\Aut(\mathbb M)$ is Roelcke precompact (point (2)).

(7) The group $\aut({\rm \bf K})$ of the lexicographically ordered square ${\bf K}$ is Roelcke precompact in the topologies $\tau_{\partial}$ and $\tau_p$, since it is easy to check that the maximal equiuniformity under the action $(\aut({\rm \bf K}), \tau_{\partial})\curvearrowright ({\rm \bf K}, \tau_d)$ is totally bounded. A different approach will be presented in \S5.}
\end{ex}

\begin{rem} {\rm In Theorem \ref{Roelcke precomp3-3} and Corollary \ref{Roelcke precomp3-4} the automorphism group of an ultrahomogeneous simple set can be replaced by any of its subgroups acting in an ultratransitive way. In particular, this is due to the fact that such subgroups are dense subgroups of the automorphism group in the permutation topology.} 
\end{rem}
 
\section{The Roelcke compactifications of subgroups of ${\rm S}(X)$} \label{R-prec_perm-top}
{\bf I.} Let $G$ be a subgroup of the permutation group $({\rm S}(X), \tau_{\partial})$  of the discrete infinite space~$X$.

\begin{thm}\label{Roelcke precomp4-1}
\noindent {\rm (1)} $L\wedge R=\mathcal U$, if for points $x_1,\ldots, x_n\in X$, $n\in\mathbb N$, and any $g, h\in G$ such that $h(x_k)\in {\rm St}_{x_1, \dots, x_n}g(x_k)$, $k=1,\ldots, n$, there exist $f\in {\rm St}_{x_1, \dots, x_n}$ and $g'\in g{\rm St}_{x_1, \dots, x_n}$ such that $h=f\circ g'$.

\noindent  {\rm (2)} Let $L\wedge R=\mathcal U$ and the uniformity $\mathcal U_X$ on $X$ is totally bounded. Then $G$ is Roelcke precompact and the Roelcke compactification of $G$ is the closure of $\imath(G)=G$ in $(\beta_G X)^X$. 

\noindent  {\rm (3)}  Let $L\wedge R=\mathcal U$ and the uniformity $\mathcal U_X$ on $X$ is totally bounded. If for any point $x\in\beta_G X\setminus X$ and any $U\in\tilde{\mathcal U}_X$ {\rm(}extension of $\mathcal U_X$ to $\beta_G X${\rm)} there exist 
\noindent points $x_1,\ldots, x_n\in X$ and  $V\in\tilde{\mathcal U}_X$ such that for any $g\in G$ 
$$\{h\in G\ |\ (g(x_k), h(x_k))\in V,\ k=1,\ldots, n\}\subset\{h\in G\ |\ (g(x), h(x))\in U\},\leqno{(\star)}$$
then $\mathcal U=\tilde{\mathcal U}$ {\rm(}$\tilde{\mathcal U}$ is the restriction of the equiuniformity on the product $(\beta_G X)^{\beta_G X}$ onto $G=\jmath(G)${\rm)} and the Roelcke compactification of $(G, \tau_{\partial})$ is the enveloping Ellis semigroup the action  $(G, \tau_{\partial})\curvearrowright\beta_G X$ {\rm (}the closure of $\jmath(G)=G$ in $(\beta_G X)^{\beta_G X}${\rm )}.
\end{thm}
 
\begin{proof} Since $\mathcal U$ is an equiuniformity on $G$ and $L\wedge R=R_{\mathcal K}$ (Proposition \ref{coin}), then to prove (1), by Theorem \ref{Roelcke precomp2-1}, it is necessary to show that  
$L\wedge R\subset \mathcal U$. 
 
For any covering
$\{{\rm St}_{x_1, \dots, x_n}g{\rm St}_{x_1, \dots, x_n}\ |\ g\in G\}, x_1,\dots, x_n\in X$, consider the covering 
$$\{U_{x_1, \dots, x_n; g; {\rm U}_{x_1, \dots, x_n}}=\{h\in G\ |\ (g(x_k), h(x_k))\in {\rm U}_{x_1, \dots, x_n}, k=1, \dots, n\}\ |\ g\in G\},$$
${\rm U}_{x_1, \dots, x_n}$ is the diagonal entourage corresponding to the covering $\{{\rm St}_{x_1, \dots, x_n}x\ |\ x\in X\}$ of the space $X$.
 
If $h\in U_{x_1, \dots, x_n; g; {\rm U}_{x_1, \dots, x_n}}$, then $ (g(x_k), h(x_k))\in {\rm U}_{x_1, \dots, x_n}$ iff $h(x_k)\in {\rm St}_{x_1, \dots, x_n} g(x_k)$, $k=1,\ldots, n$. By the condition of the theorem, there exist $f\in {\rm St}_{x_1, \dots, x_n}$ and $g'\in g{\rm St}_{x_1, \dots, x_n}$ such that $h=f\circ g'$ $\Longrightarrow$ $h\in {\rm St}_{x_1, \dots, x_n}g{\rm St}_{x_1, \dots, x_n}$ $\Longrightarrow$ $U_{x_1, \dots, x_n; g; {\rm U}}\subset {\rm St}_{x_1, \dots, x_n}g{\rm St}_{x_1, \dots, x_n}$ $\Longrightarrow$ the covering $\{U_{x_1, \dots, x_n; g; {\rm U}}\ |\ g\in G\}$ is refined in the covering $\{{\rm St}_{x_1, \dots, x_n}g{\rm St}_{x_1, \dots, x_n}: g\in G\}$ $\Longrightarrow$ $L\wedge R\subset \mathcal U$. 
 
Additionally, we note that  
$$\{{\rm St}_{x_1, \dots, x_n}g{\rm St}_{x_1, \dots, x_n}\ |\ g\in G\}=\{U_{x_1, \dots, x_n; g; {\rm U}_{x_1, \dots, x_n}}\ |\ g\in G\}.$$

(2) If the uniformity $\mathcal U_X$ is totally bounded then $\mathcal U$ is totally bounded $\stackrel{L\wedge R=\mathcal U}{\Longrightarrow}$ $G$ is Roelcke precompact. 

Since the completion $\tilde X^{\mathcal U_X}$ is $\beta_G X$,  $L\wedge R=\mathcal U$, then from setting a totally bounded uniformity $\mathcal U_{\Pi}$ on the product $X^X$ and the equality $\mathcal U=\mathcal U_{\Pi}|_{\imath(G)}$ the last statement in (2) follows.
 
(3) The restriction $p$ of the projection of ${\rm pr}:(\beta_G X)^{\beta_G X}\to (\beta_G X)^X$ (which is uniformly continuous) onto $G=\jmath (G)$ is a bijection of $\jmath (G)$ onto $\imath(G)$ where  $\imath: G\to \beta_G X$, $\imath (g)=(g(x))_{x\in X}$ (since elements of $G$ are contionuous). Thus, after identification of $\imath (G)$ and $\jmath (G)$ , it follows that $\mathcal U\subset\tilde{\mathcal U}$.

From the condition $(\star)$ it follows that for any point $x\in\beta_G X\setminus X$ and any $U\in\tilde{\mathcal U}_X$ there exist points $x_1,\ldots, x_n\in X$ and  $V\in\mathcal U_X$ (note that $g(x_k)\in X$, $k=1,\ldots, n$, and the restriction of $\tilde{\mathcal U}_X$ onto $X$ is $\mathcal U_X$) such that the covering 
$$\{\{h\in G\ |\ (g(x_k), h(x_k))\in V,\ k=1,\ldots, n\}\ |\ g\in G\}\in\mathcal U_X$$
is refined in
$$\{\{h\in G\ |\ (g(x), h(x))\in U\}\ \ g\in G\}\in\tilde{\mathcal U}_X.$$
Hence, $p$ is a uniform equivalence and $\mathcal U=\tilde{\mathcal U}$.

$L\wedge R=\mathcal U=\tilde{\mathcal U}$.  Since,  $(\beta_G X)^{\beta_G X}$ is compact, the closure of $\jmath(G)=G$ in $(\beta_G X)^{\beta_G X}$ (completion with respect to $\tilde{\mathcal U}$) is the Roelcke compactification of $(G, \tau_{\partial})$ and is the enveloping Ellis semigroup of $G$.
\end{proof}

\begin{cor}\label{Roelcke precomp4-2}
If the action of the group $G$ on the discrete space $X$ is ultratransitive, then 

{\rm (1)}  the uniformity $\mathcal U_X$ is totally bounded; 

{\rm (2)} the completion $\tilde X^{\mathcal U_X}$ is the one-point Alexandroff compactification $\alpha X$; 

{\rm (3)} $L\wedge R=\mathcal U$; 

{\rm (4)}  the group $G$ is Roelcke precompact; 

{\rm (5)}  the Roelcke compactification of $G$ is  the enveloping Ellis semigroup {\rm (}the closure of $\jmath(G)=G$ in $(\alpha X)^{\alpha X}${\rm )} and  coincides with the Roelcke compactification of ${\rm S} (X)$.
\end{cor}
 
\begin{proof} (1) Due to the ultratransitivity of the action, the base of the equiuniformity $\mathcal U_X$ is formed by the disjoint coverings  
$$\{{\rm St}_{x_1, \dots, x_n}x\ |\ x\in X\}=\{\{x_1\},\ldots, \{x_n\}, X\setminus\{x_1,\ldots, x_n\}\},$$
$x_1,\dots, x_n\in X$, of  $n+1$ sets, $n\in\mathbb N$, and $\mathcal U_X$ is a totally bounded equiuniformity. 
 
(2) As a base for the uniformity on the compactum $\alpha X$, one can choose coverings 
$$\{\{x_1\},\ldots, \{x_n\}, \alpha X\setminus\{x_1,\ldots, x_n\}\},$$
$x_1, \dots, x_n\in X$, $n\in\mathbb N$. $\{\{x_1\},\ldots, \{x_n\}, \alpha X\setminus\{x_1,\ldots, x_n\}\}\wedge X$ is the base of the uniformity $\mathcal U_X$. Hence $\alpha X=\tilde X^{\mathcal U_X}$.
 
To prove (3), we check if the condition of point (1) of Theorem \ref{Roelcke precomp4-1} is met. 
 
For points $x_1,\ldots, x_n\in X$, $n\in\mathbb N$, and any $g, h\in G$ such that $h(x_k)\in {\rm St}_{x_1, \dots, x_n}g(x_k)$, $k=1,\ldots, n$, we consider, without  loss of generality, the following. 
 
If there exists a point $x_k$ such that $g(x_k)=x_{m(k)}$, $m(k)\leq n$, then let $x_1,\ldots, x_p$, $p\leq n$, be all points for which $g(x_k)=x_{m(k)}$, $m(k)\leq n$. Then $h(x_k)=x_{m(k)}$, $k=1,\ldots, p$. If $p\ne n$, then $g(x_{p+l})\not\in\{x_1,\ldots, x_n\}$,  $h(x_{p+l})\not\in\{x_1,\ldots, x_n\}$, $l=1,\ldots, n-p$.
 
If $p=n$, then it follows from the ultratransitivity of the action that there exists $f\in G$ such that the points $x_1, \dots, x_n$ are mapped respectively to the points $x_1, \dots, x_n$.
 
If $0<p<n$, then it follows from the ultratransitivity of the action that there exists $f\in G$ such that the points $x_1, \dots, x_n, g(x_{p+1}),\ldots, g(x_n)$ are mapped respectively to the points $x_1, \dots, x_n, h(x_{p+1}),\ldots, h(x_n)$. 
 
If $p=0$, then it follows from the ultratransitivity of the action that there exists $f\in G$ such that the points $x_1, \dots, x_n, g(x_{1}),\ldots, g(x_n)$ are mapped respectively to the points $x_1, \dots, x_n, h(x_{1}),\ldots, h(x_n)$. 
 
In all cases
$$f\in {\rm St}_{x_1, \dots, x_n}.$$ At the same time
$$h(x_k)=f(g(x_k)), k=1,\ldots, n.$$
Hence $h^{-1}f g\in  {\rm St}_{x_1, \dots, x_n}$ $\Longrightarrow$ $f g\in h {\rm St}_{x_1, \dots, x_n}$  $\Longrightarrow$ $h\in f g {\rm St}_{x_1, \dots, x_n}$ and there is $g'\in g{\rm St}_{x_1, \dots, x_n}$ such that $h=f g'$. By Theorem 2.1 $L\wedge R=\mathcal U$. 
 
(1), (3) and p. (2) of Theorem \ref{Roelcke precomp3-3} imply  (4). (2), (3), p. (3) of Theorem \ref{Roelcke precomp4-1} (which can be easily verified) and Fact 2 (3) imply (5).
\end{proof}

{\bf II.} Let $X$ be an ultrahomogeneous chain, $\aut(X)$ is its automorphism group. 

$\mathcal U_X^{p}$is the maximal equiuniformity for the action $(\aut(X), \tau_{p})\curvearrowright (X, \tau)$,   
$\beta_{p}X=m (X, \tau)$ is the maximal $G$-compactification of $X$ by Lemma \ref{comp}(1) (completion of $X$ with respect to the uniformity $\mathcal U_X^{p}$ \cite[Theorem 3]{ChK2}), $\mathcal U^{p}$ is a totally bounded equiuniformity on $G\subset X^X$, constructed in \S\ref{Ellis_constr}. 

$\mathcal U_X^{\partial}$ is the maximal equiuniformity for the action $(\aut(X), \tau_{\partial})\curvearrowright (X, \tau_d)$,  
$\beta_{\partial}X=m (X, \tau_d)$ is the maximal $G$-compactification of $X$ by Lemma \ref{comp}(2)  (completion of $X$ with respect to the uniformity $\mathcal U_X^{\partial}$ \cite[Theorem 3]{ChK2}),
$\mathcal U^{\partial}$ is a totally bounded equiuniformity on $G\subset X^X$, constructed in \S\ref{Ellis_constr}. 

$R_{\mathcal K}$ is a uniformity on $\aut(X)$, constructed from a family of small subgroups $\mathcal K=\{{\rm St}_{x_1, \dots, x_n}\ |\  x_1,\dots, x_n\in X\}$.

\begin{cor}\label{Roelcke precomp4-3} 
Let $X$ be an ultrahomogeneous chain. Then 

{\rm(1.1)} for the group $(\aut(X), \tau_{\partial})$ one has $L\wedge R=\mathcal U^{\partial}$, the Roelcke compactification of $G=(\aut(X), \tau_{\partial})$ is the closure of $\imath(G)=G$ in $(\beta_{\partial}X)^{X}$,

{\rm(1.2)} for the group $(\aut(X), \tau_p)$  one has $R_{\mathcal K}=\mathcal U^p$, the completion of  $G=(\aut(X), \tau_p)$ with respect to the  uniformity $R_{\mathcal K}$ is the closure of $\imath(G)=G$ in $(\beta_pX)^{X}$.

Let $X$ be a continuous ultrahomogeneous chain. Then 

{\rm(2.1)} the Roelcke compactification of $G=(\aut(X), \tau_{\partial})$ is the enveloping Ellis semigroup {\rm (}the closure of $\jmath(G)=G$ in $(\beta_{\partial}X)^{\beta_{\partial} X}${\rm )},

{\rm(2.2)} the completion of the group $G=(\aut(X), \tau_p)$ with respect to the  uniformity $R_{\mathcal K}$ is the enveloping Ellis semigroup {\rm (}the closure of $\jmath(G)=G$ in $(\beta_pX)^{\beta_p X}${\rm )}.
\end{cor} 
 
\begin{proof}
(1.1) By Theorem \ref{Roelcke precomp4-1}, at first we establish the equality $L\wedge R=\mathcal U^{\partial}$, i.e. check if the condition (1) of Theorem \ref{Roelcke precomp4-1} is met.

Let $x_1<\ldots<x_n\in X$, $n\in\mathbb N$, and $g, h\in G$ be such that $h(x_k)\in {\rm St}_{x_1, \dots, x_n}g(x_k)$, $k=1,\ldots, n$. The covering $\{\st_{x_1,\ldots, x_n}x\ |\ x\in X\}$ has the form
$$(\gets, x_1)\cup\{x_1\}\cup(x_1, x_2)\cup\{x_2\}\cup\dots\cup(x_{n-1}, x_n)\cup\{x_n\}\cup(x_n, \to), x_1<\ldots <x_n,$$
and for any $k=1,\ldots, n$ the points $h(x_k), g(x_k)$ belong to one of its disjoint elements. Since 
$$g(x_1)<\ldots<g(x_n)~ \mbox{and}~ h(x_1)<\ldots<h(x_n),$$
then arranging in order the points $x_1<\ldots<x_n$ and $g(x_1)<\ldots<g(x_n)$, $x_1<\ldots<x_n$ and $h(x_1)<\ldots<h(x_n)$,  we will get ordered sets of points $P_g$ and $P_h$ for which

first, $x_j\leq g(x_k)\leq x_{j+1}\Longleftrightarrow x_j\leq h(x_k)\leq x_{j+1}$, $k=1,\ldots, n$, and

second, $x_j< g(x_k)<g(x_m)<x_{j+1}\Longleftrightarrow x_j<h(x_k)<h(x_m)<x_{j+1}$, $k<m$, $k, m=1,\ldots, n$.

The ultratransitivity of the action implies the existence of $f\in\aut(X)$ such that $f(P_g)=P_h$. Then $f\in\St_{x_1, \dots, x_n}$ and $h=f\circ g$.

The results of \S\ref{M-G-comp} and item (2) of Theorem~\ref{Roelcke precomp4-1} finishes the proof.

In order to prove (2.1) it is enough to check the fulfillment of the condition $(\star)$ of Theorem~\ref{Roelcke precomp4-1}. Let, for example, $(x, 1)\in  m(X, \tau_d)=\beta_{\partial} X$ and an arbitrary covering of  $\{(x, 1)\}\times m(X, \tau_d)$  be of the form 
$$\Omega=(\gets, (x_1, 0))\bigcup\{(x_1, 0)\}\bigcup ((x_1, 0), (x_2, 0))\bigcup\{(x_2, 0)\}\bigcup\ldots$$
$$\ldots\bigcup((x_{n-1}, 0), (x_n, 0))\bigcup\{(x_n, 0)\}\bigcup ((x_n, 0), \to),$$
where $x_1<x_2<\ldots<x_{n-1}<x_n$. It consists of disjoint sets. 
Let us note that if $(y, 0)\in (\gets, (x_1, 0))$ or $((x_k, 0), (x_{k+1}, 0))$, $k=1,\ldots, n-1$, or $ ((x_n, 0), \to)$, then  $(y, -1), (y, 1)\in (\gets, (x_1, 0))$ or $ ((x_k, 0), (x_{k+1}, 0))$, $k=1,\ldots, n-1$, or $ ((x_n, 0), \to)$ also. 

Take  $(x, 0)\in m(X, \tau_d)$ and the same covering $\Omega$ on  $\{(x, 0)\}\times m(X, \tau_d)$. 

If $f(x, 0)=g(x, 0)$, then  $f(x, 1)=g(x, 1)$.  In particular, if $(x_k, 1)$ is defined and  $f(x, 0)=g(x, 0)=(x_k, 0)$, then $f(x, 1)=g(x, 1)=(x_k, 1)\in ((x_k, 0), (x_{k+1}, 0))$, $k=1,\ldots, n$.

If $f(x, 0)=(y, 0), g(x, 0)=(z, 0)\in (\gets, (x_1, 0))$ or $ ((x_k, 0), (x_{k+1}, 0))$, $k=1,\ldots, n-1$, or $ ((x_n, 0), \to)$, then $f(x, 1)=(y, 1)$, $g(x, 1)=(z, 1)$. Pair of points $(y, 0), (y, 1)$ (and $(z, 0), (z, 1)$) is in one element of the covering $\Omega$, pair of points  $(y, 0), (z, 0)$ is in one element of the covering $\Omega$. Hence,   pair of points $f(x, 1)=(y, 1)$, $g(x, 1)=(z, 1)$ is in one element of the covering $\Omega$. Therefore the condition  $(\star)$ of Theorem~\ref{Roelcke precomp4-1} is valid. 

In the case when $(x, -1)\in m(X, \tau_d)$ the same resoning can be used. In the points $\{\inf\}$ and $\{\sup\}\in m(X, \tau_d)$  the fulfillment of the condition $(\star)$ of Theorem~\ref{Roelcke precomp4-1} is evident. 

(1.2) In order to establish the equality $R_{\mathcal K}=\mathcal U^p$ it is enough to check if the condition (2) of Theorem {\rm\ref{Roelcke precomp2-1}} is met. 

For the diagonal entourage $U\in\mathcal U_X^{p}$ let the diagonal entourage $V\in\mathcal U_X^{p}$ be such that $2V\subset U$; $x_1<\ldots<x_n\in X$; $g, h\in\aut(X)$ such that $(g(x_k), h(x_k))\in V$, $k=1,\ldots, n$. Let us construct an automorphism $g'\in g{\rm St}_{x_1, \dots, x_n}$ such that $(h(x), g'(x))\in 2V\subset U$ for any point $x\in X$ (then for any $x\in X$ $(h((g')^{-1}(x)), g'((g')^{-1}(x)))=(h(g')^{-1}(x), x)\in U$ and from the definition of the topology of pointwise convergence, the automorphism $h(g')^{-1}$ belongs to the neighbourhood of the unit of the group $(\aut(X), \tau_p)$, which determines the entourage of the diagonal $U$, i.e.  the condition (2) of Theorem {\rm\ref{Roelcke precomp2-1}} is met). The construction of the automorphism $g'$ is carried out on the intervals $$(\gets, x_1], [x_1, x_2],\ldots, [x_{n-1}, x_n], [x_n, \to).$$ 

(i) On the interval $(\gets, x_1]$ in the case of $h(x_1)\geq g(x_1)$. If $(\gets , h(x_1))\in V$, then $g'(x)=g(x)$ for $x\in (\gets, x_1]$. 

Otherwise, let $a_1<g(x_1)$, $(a_1, g(x_1))\in V$. We assume $x_1^-=h^{-1}(a_1)$. From the ultratransitivity of the action, there exists $\varphi_1\in\aut(X)$ such that $\varphi_1(x_1^-)=a_1$, $\varphi_1(x_1)=g(x_1)$. Let us put
$$g'(x)=\left\{\begin{array}{lcl}
h(x) & \mbox{if} & x\in (\gets, x_1^-], \\
\varphi_1(x) & \mbox{if} & x\in [x_1^-, x_1].\\
\end{array}\right.$$ 
In the case of $h(x_1)\leq g(x_1)$.  If $(\gets , g(x_1))\in V$, then $g'(x)=g(x)$ for $x\in (\gets, x_1]$. 

Otherwise, let $a_1<h(x_1)$, $(a_1, h(x_1))\in V$.  We assume $x_1^-=h^{-1}(a_1)$. From the ultratransitivity of the action, there exists $\varphi_1\in\aut(X)$ such that $\varphi_1(x_1^-)=a_1$, $\varphi_1(x_1)=g(x_1)$. Let us put
$$g'(x)=\left\{\begin{array}{lcl}
h(x) & \mbox{if} & x\in (\gets, x_1^-], \\
\varphi_1(x) & \mbox{if} & x\in [x_1^-, x_1].\\
\end{array}\right.$$ 

On the interval $[x_n, \to)$ the construction is similar. 

(ii) On the interval $[x_k, x_{k+1}]$, $k=1, \dots, n-1$. The case of $h(x_k)\geq g(k_k)$,  $h(k_{k+1})\geq g(x_{k+1})$. If  $(g(x_k), h(x_{k+1}))\in V$, then 
$g'(x)=g(x)$ for $x\in [x_k, x_{k+1}]$.

Otherwise, let $b_k>h(x_k)$, $(b_k, g(x_k))\in V$, $a_{k+1}<g(x_{k+1})$, $(a_{k+1}, h(x_{k+1}))\in V$ ($b_k<a_{k+1}$). We assume $x_k^+=h^{-1}(b_k)$, $x_{k+1}^-=h^{-1}(a_{k+1})$. From the ultratransitivity of the action, there exists $\varphi_k^-\in\aut(X)$ such that $\varphi_k^-(x_k)=g(x_k)$, $\varphi_k^-(x_k^+)=b_k$, there exists $\varphi_k^+\in\aut(X)$ such that $\varphi_k^+(x_{k+1}^-)=a_{k+1}$, $\varphi_k^+(x_k)=g(x_{k+1})$. Let us put
$$g'(x)=\left\{\begin{array}{lcl}
\varphi_k^- & \mbox{if} & x\in [x_k, x_k^+], \\
h(x) & \mbox{if} & x\in [x_k^+, x_{k+1}^-], \\
\varphi_k^+ & \mbox{if} & x\in [x_{k+1}^-, x_{k+1}].\\
\end{array}\right.$$ 
In the case of $h(x_k)\leq g(x_k)$,  $h(x_{k+1})\leq g(x_{k+1})$  the construction is similar. 

In the case of $h(x_k)\geq g(x_k)$,  $h(x_{k+1})\leq g(x_{k+1})$.   If  $(g(x_k), g(x_{k+1}))\in V$, then $g'(x)=g(x)$ for $x\in [x_k, x_{k+1}]$.

Otherwise, let $b_k>h(x_k)$, $(b_k, g(x_k))\in V$, $a_{k+1}<h(x_{k+1})$, $(a_{k+1}, g(x_{k+1}))\in V$ ($b_k<a_{k+1}$). We assume $x_k^+=h^{-1}(b_k)$, $x_{k+1}^-=h^{-1}(a_{k+1})$. From the ultratransitivity of the action, there exists $\varphi_k^-\in\aut X$ such that $\varphi_k^-(x_k)=g(x_k)$, $\varphi_k^-(x_k^+)=b_k$, there exists $\varphi_k^+\in\aut(X)$ such that $\varphi_k^+(x_{k+1}^-)=a_{k+1}$, $\varphi_k^+(x_k)=g(x_{k+1})$. Let us put
$$g'(x)=\left\{\begin{array}{lcl}
\varphi_k^- & \mbox{if} & x\in [x_k, x_k^+], \\
h(x) & \mbox{if} & x\in [x_k^+, x_{k+1}^-], \\
\varphi_k^+ & \mbox{if} & x\in [x_{k+1}^-, x_{k+1}].\\
\end{array}\right.$$ 
In the case of $h(x_k)\leq g(x_k)$, $h(x_{k+1})\geq g(x_{k+1})$ the construction is similar. It follows from the construction that $(h(x), g'(x))\in 2V\subset U$ for any point $x\in X$.

To proof  (2.2) we note that in the points $\{\inf\}$ and $\{\sup\}\in m(X, \tau_d)$ the condition $(\star)$ of Theorem~\ref{Roelcke precomp4-1} is evidently fulfilled. 
\end{proof}

\begin{rem} {\rm (1) Since $L\wedge R\subset R_{\mathcal K}$, then under the conditions of {\rm Corollary \ref{Roelcke precomp4-3}} (2) the Roelcke compactification of $G=(\aut(X), \tau_p)$ is the image of completion (compactification) of the group $G=(\aut(X), \tau_p)$ with respect to the  uniformity $R_{\mathcal K}$ when mapping compactifications. 

(2) In Corollary \ref{Roelcke precomp4-3} the automorphism group of an ultrahomogeneous simple set can be replaced by any of its subgroups acting in an ultratransitive way. In particular, this is due to the fact that such subgroups are dense subgroups of the automorphism group in the permutation topology.} 
\end{rem}

\section{The Roelcke precompactness of automorphism groups of homogeneous chains different from simple ones}
Let $X$ be a homogeneous chain, $J$ is a proper regular interval. From \cite[Theorem 7]{Ohkuma2} it follows that $J$ is an open interval $(X, \tau)$, the complement to which are disjoint open intervals $J^-=\{t\in X\ |\ t<x, \forall x\in J\}$ and $J^+=\{t\in X\ |\ t>x, \forall x\in J\}$ (not necessarily homogeneous). 

By Fact 4(3) $J$ is a homogeneous chain and the group $\aut(J)$ (any automorphism $J$ extends to an automorphism of $X$ coinciding with the identical mapping on $X\setminus J$) is called a characteristic group of $J$ \cite[p. (3.1)]{Ohkuma2}.

Assume $H=\{g\in\aut(X) | \forall x\in J, g(x)\in J\}$. 

\begin{pro}\label{Roelcke precomp5-1}

{\rm (1)} $H$ is an open subgroup of $(\aut(X), \tau_p)$ {\rm(}and hence $(\aut(X), \tau_{\partial})${\rm)}.

{\rm (2)} 
$(H, \tau_p)=(\aut(J^-), \tau_p)\times (\aut(J), \tau_p)\times  (\aut(J^+), \tau_p),$

\quad $(H, \tau_{\partial})=(\aut(J^-), \tau_{\partial})\times (\aut(J), \tau_{\partial})\times (\aut(J^+), \tau_{\partial}),$
 
{\rm (3)} for any $x\in J$ $$\st_x {\footnotesize(\mbox{under the action}~\aut(X)\curvearrowright X)} =\aut(J^-)\times\st_x {\footnotesize(\mbox{under the action}~\aut(J)\curvearrowright J)} \times \aut(J^+).$$
\end{pro}

\begin{proof} (1) $H=[x, J]$, where $x\in J$. So $H$ is an open subgroup of $(\aut(X), \tau_p)$. 

(2) $X=J^-\lozenge J\lozenge J^+$, and any automorphism $X$ belongs to $H$ iff it is a combination of automorphisms  $J^-$, $J$ and $J^+$. Thus, $H=\aut(J)\times\aut(J^-)\times\aut(J^+)$.

Since the topology of pointwise convergence is an admissible group topology on the automorphism group of a LOTS \cite{Ovch} and \cite{Sorin} and $\tau_p\leq \tau_{\partial}$, then the topological groups in point (2) are correctly defined. The topological isomorphism of the groups in point (2) is easily verified.

(3) follows from (2) and the definition of a regular interval.
\end{proof}

If $J$ is a regular interval, then the equivalence relation $$x\sim_J y\Longleftrightarrow\forall g\in\aut(X)~ ((g(x)\in J)\Longrightarrow (g(y)\in J))$$
defines a homogeneous chain $X/J=X/\sim_J$ \cite[\S4]{Ohkuma2}. 

\begin{thm}\label{LexO_Aut}
Let $J$ be the proper regular interval of a homogeneous chain $X$. Then

{\rm (1)} $\aut(X)=(\aut(J))^{X/J}\ltimes\aut(X/J)$. 

{\rm (2)} $(\aut(X), \tau_p)=(\aut(J),  \tau_p)^{X/J}\ltimes (\aut(X/J), \tau_{\partial})$. 

{\rm (3)} $(\aut(X), \tau_{\partial})=(\aut(J),  \tau_{\partial})^{X/J}\ltimes (\aut(X/J), \tau_{\partial})$. 

{\rm (4)} $(\aut(X), \tau_p)/(\aut(J), \tau_p)^{X/J}=(\aut(X), \tau_{\partial})/(\aut(J), \tau_{\partial})^{X/J}=(\aut(X/J), \tau_{\partial})$.
\end{thm}

\begin{proof} (1) and (2) are a restatement of Theorem 8 from \cite{Sorin}.

For a point $x\in X$ by $J_x\in X/J$, we denote its equivalence class with respect to $\sim_J$, $J_x$ and $J$ are isomorphic chains, $p_x: (\aut(J))^{X/J}\to \aut(J_x)=\aut(J)$ projection of the product on the factor.

(3) follows from the coincidence of the subbase neighbourhoods: for $x\in X$, the neighbourhood $\st_x$ from the subbase (under the action $\aut(X)\curvearrowright X$) coincides with the set $p_x^{-1}\st_x (\mbox{under the action}~\aut(J)\curvearrowright J)\times\st_{J_x}$, where $x\in J_x$, from the topology subbase $(\aut(J),  \tau_{\partial})^{X/J}\ltimes (\aut(X/J), \tau_{\partial})$, and vice versa.

(4) corollary of (2) and (3).
\end{proof}

\begin{thm}\label{Roelcke precomp5-2} 
Let $X$ be a homogeneous chain that is not simple. The following conditions are equivalent:

{\rm (1)}  the topological group $(\aut (X), \tau_p)$ {\rm\big(}respectively {\rm$(\aut(X), \tau_{\partial})$\big)} is Roelcke precompact,

{\rm (2)}  for any proper regular interval $J$ topological groups $(\aut (J),  \tau_p)$ {\rm\big(}respectively\linebreak {\rm$(\aut(J), \tau_{\partial})$\big)} and $(\aut (X/J), \tau_{\partial})$  are Roelcke precompact,

{\rm (3)}  there exists a proper regular interval $J$ such that the topological groups $(\aut (J),  \tau_p)$ {\rm\big(}respectively {\rm$(\aut(J), \tau_{\partial})$\big)} and $(\aut (X/J), \tau_{\partial})$ are Roelcke precompact.
\end{thm} 

\begin{proof} (1)$\Longrightarrow$(2). Point (4) of  Theorem \ref{LexO_Aut} and point (3) of Fact 1 imply the Roelcke precompactness $(\aut (X/J), \tau_{\partial})$.  Points (1) and (2) of  Proposition \ref{Roelcke precomp5-1} and points (2) and (6) of Fact 1 imply the Roelcke precompactness $(\aut (J),  \tau_p)$ \big(respectively $(\aut(J), \tau_{\partial})$\big).

The implication (2)$\Longrightarrow$(3) is obvious.

(3)$\Longrightarrow$(1) follows from point (4) of Fact 1.
\end{proof}

From Corollary \ref{Roelcke precomp3-4}  and Theorem \ref{Roelcke precomp5-2} we have

\begin{cor}\label{coin2}
For a homogeneous chain $X$ having a simple infinite interval, the conditions of the Roelcke precompactness of groups $(\aut (X), \tau_p)$ and $(\aut(X), \tau_{\partial})$ are equivalent.\hfill $\square$
\end{cor}

\begin{ex} {\rm (1)} The automorphism group of a lexicographically ordered product of two sets with $2$-homogeneous factors is Roelcke precompact in the topologies $\tau_p$ and $\tau_{\partial}$. 

In particular, automorphism groups of sets of the form $X_1\otimes_{\ell} X_2$, where $X_1,  X_2\in\{\mathbb Q, \mathbb P, \mathbb R, L, \tilde L\}$, are Roelcke precompact in topologies $\tau_p$ and $\tau_{\partial}$. 

The automorphism group of a lexicographically ordered square {\bf K} is Roelcke precompact in the topologies $\tau_p$ and $\tau_{\partial}$, since $\aut({\bf K})$ is isomorphic to the group $\aut(\mathbb R)\times\big((\aut(\mathbb R)^{\mathbb R}\ltimes\aut(\mathbb R))\times\aut(\mathbb R)$ and $\aut(\mathbb R)^{\mathbb R}\ltimes\aut(\mathbb R)$ is the  automorphism group of $\mathbb R\otimes_{\ell} \mathbb R$ {\rm\cite[Corollary 5]{Sorin}}.

{\rm(2)} Automorphism groups of lexicographically ordered products of two sets, where the second factor is a rigid set, are not Roelcke precompact in the topologies $\tau_p$ and $\tau_{\partial}$. 

In particular, the automorphism groups of the set $X\otimes_{\ell}\mathbb Z$ are not Roelcke precompact in the topologies $\tau_p$ and $\tau_{\partial}$. 

The automorphism group of a homogeneous discrete set {\rm\cite[Definition 4]{Ohkuma1}}  is not Roelcke precompact in the topologies $\tau_p$ and $\tau_{\partial}$.
\end{ex}

\begin{thm}\label{Roelcke precomp5-3} 
Let there be no simple proper regular intervals in a homogeneous chain $X$ that is not simple. Then $(\aut(X), \tau_p)$ is Roelcke precompact iff the automorphism groups $(\aut(X/J), \tau_{\partial})$, with $J$ being regular intervals, are Roelcke precompact.
\end{thm} 

\begin{proof} Necessity follows from Theorem \ref{Roelcke precomp5-2}.

Sufficiency. For an arbitrary point $x\in X$, the family $S$ of all proper regular intervals containing $x$ is linearly ordered by inclusion \cite[p. 4.6]{Ohkuma2}. Their intersection is also a regular interval \cite[p. 4.7]{Ohkuma2}. If it is a proper regular interval, then it is simple (a contradiction with the condition of the theorem). Therefore, the intersection is one-point, and the open-closed proper regular intervals form the base of the topology of linear order at the point~$x$. 

For any $J\in S$, the following is defined: the quotient map $f_J:X\to X/J$ and the homomorphism $\varphi_J:(\aut(X), \tau_p)\to  (\aut(X/J), \tau_{\partial})$ by point (2) of Theorem \ref{LexO_Aut}. In this case, the pair of mappings $$(\varphi_J, f_J): ((\aut(X), \tau_p), X)\to ((\aut(X/J), \tau_{\partial}), X/J)$$ is an equivariant pair of mappings (natural actions are omitted in the notation, see \cite[p. 2.1]{Kozlov2017}). Indeed, $f_J(t)=J_t$ is a regular interval containing the point $t$, $g(J_t)=J_{g(t)}$ is a regular interval containing the point $g(t)$. Hence 
$$\varphi_J (g) (f_J(t))=J_{g(t)}=f_J(g(t)).$$

Note that for $J'>J$ $f_{J'}(J)$ is a regular interval in $X/J'$. Indeed, $J'_x, J'_y\in f_{J'}(J)$, $g'\in\aut(X/J)'$, $g'(J'_x)\in f_{J'}(J)$. Let $\varphi_{J'}(g)=g'$, $x, y\in J$ and $g(x)\in J$. Then $g(y)\in J$, $g'(J'_y)=f_{J'}(g(y))\in  f_{J'}(J)$.

For $J'>J$, an equivariant pair of mappings is defined $$(\varphi_{J'J}, f_{J'J}): ((\aut(X/J'), \tau_{\partial}), X/J')\to ((\aut(X/J), \tau_{\partial}), X/J),$$
such that $\varphi_{J'J}\circ\varphi_{J'}=\varphi_{J}$ (since $(\aut(X/J'), \tau_{\partial})=(\aut(f_{J'}(J)),  \tau_{\partial})^{(X/J')/f_{J'}(J)}\ltimes\linebreak (\aut((X/J')/f_{J'}(J)), \tau_{\partial})=(\aut(f_{J'}(J)),  \tau_{\partial})^{X/J}\ltimes (\aut(X/J), \tau_{\partial})$),  $f_{J'J}\circ f_{J'}=f_{J}$ ($X/J=(X/J')/f_{J'}(J)$). 

This define the equivariant inverse spectrum $$\big\{((\aut(X/J), \tau_{\partial}), X/J),  (\varphi_{J'J}, f_{J'J}), S\big\}$$
and the group spectrum $\big\{(\aut(X/J), \tau_{\partial}), \varphi_{J'J}, S\big\}$ \cite[p. 2.5]{Kozlov2017}, in the inverse limit of which the group $(\aut(X), \tau_p)$ is dense, if the family of homomorphisms $\varphi_{J}, J\in S$, separates points and closed sets. 

Any neighbourhood of the unit of the group $(\aut(X), \tau_p)$ has the form 
$$O=[x_1, O_1]\cap\ldots\cap [x_n, O_n].$$
Let the regular interval $J\in S$ be such that $J_{x_k}\subset O_k$, $k=1, \ldots, n$. Then $\varphi^{-1}_J(\st_{J_{x_1}}\cap\ldots\cap \st_{J_{x_n}})\subset O$ and $\st_{J_{x_1}}\cap\ldots\cap \st_{J_{x_n}}$ is the neighbourhood of the unit of the group $(\aut(X/J), \tau_{\partial})$. Thus $(\aut(X), \tau_p)$ is a dense subgroup of the Roelcke precompact group by point (5) of Fact 1 and is Roelcke precompact by point (1) of Fact 1.
\end{proof} 

{\bf Question.} Let $X$ be a homogeneous chain. Are the conditions of the Roelcke precompactness of the group $(\aut(X), \tau_{\partial})$ and the Roelcke precompactness of the group $(\aut(X), \tau_p)$ equivalent?

\end{document}